\documentclass[12pt,oneside,english,shortlabels]{amsart}

\usepackage{babel}
\usepackage{amstext}
\usepackage{amssymb}

\usepackage{amsmath} 
\usepackage{latexsym}
\usepackage{graphicx} 

\usepackage{mathptmx}

\usepackage{graphicx} 

\makeatletter

\usepackage{amsthm}
\usepackage{enumitem}		

\usepackage{mathrsfs} 
\providecommand{\noopsort}[1]{}

\theoremstyle{definition} 

 \newtheorem{definition}{Definition}[section]
 \newtheorem{remark}[definition]{Remark}

\newtheorem*{notation}{Notations}

\theoremstyle{plain}      

 \newtheorem{proposition}[definition]{Proposition}
 \newtheorem{theorem}[definition]{Theorem}
 \newtheorem{corollary}[definition]{Corollary}
 \newtheorem{lemma}[definition]{Lemma}

\newtheorem{conjecture}{Conjecture}

\makeatletter
\newcommand*{\house}[1]{
  \mathord{
    \mathpalette\@house{#1}
  }
}
\newcommand*{\@house}[2]{
  \dimen@=\fontdimen8 %
      \ifx#1\scriptscriptstyle\scriptscriptfont
      \else\ifx#1\scriptstyle\scriptfont
      \else\textfont\fi\fi
      3 %
  \sbox0{%
    $#1%
      \vrule width\dimen@\relax
      \overline{%
        \kern2\dimen@
        \begingroup 
          #2%
        \endgroup
        \kern2\dimen@
      }
      \vrule width\dimen@\relax
      \mathsurround=1.5\dimen@ 
    $
  }
  \ht0=\dimexpr\ht0-\dimen@\relax
  \dp0=\dimexpr\dp0+2\dimen@\relax
  \vbox{
    \kern\dimen@ 
    \copy0 
  }
}

\def\dyg{{\rm dyg}}

\newcommand{\lo}{{\rm Log\,}}

\newcommand{\rb}{\mathbb{R}}
\newcommand{\pb}{\mathbb{P}}

\newcommand{\cb}{\mathbb{C}}

\newcommand{\zb}{\mathbb{Z}}

\newcommand{\qb}{\mathbb{Q}}
\newcommand{\nb}{\mathbb{N}}

\def\11{{\mathbf 1}}

\theoremstyle{remark}
\newtheorem{exampl}[subsubsection]{Example}

\def\bee{\begin{exampl}}
\def\eee{\end{exampl}}
\def\bn{\begin{notation}}
\def\en{\end{notation}}
\def\br{\begin{remark}}
\def\er{\end{remark}}
\def\bp{\begin{prop}}
\def\ep{\end{prop}}
\def\bpr{\begin{proof}}
\def\epr{\end{proof}}
\def\bt{\begin{thm}}
\def\et{\end{thm}}
\def\be{\begin{equation}}
\def\ee{\end{equation}}
\def\bl{\begin{lem}}
\def\el{\end{lem}}
\def\bc{\begin{cor}}
\def\ec{\end{cor}}
\def\bd{\begin{defn}}
\def\ed{\end{defn}}

\def\dyg{{\rm dyg}}

\numberwithin{equation}{section}

\begin{document}
\title[A non-trivial minoration for the set of Salem numbers]{A non-trivial minoration for the set of Salem numbers}
\author{Jean-Louis Verger-Gaugry}
\thanks{}
\address{
Univ. Grenoble Alpes, Univ. Savoie Mont~Blanc,
LAMA, CNRS UMR 5127,
F - \!73000 Chamb\'ery, \!France}
\email{verger-gaugry.jean-louis@unsechat-math.fr}

\begin{abstract}
The set of Salem numbers is proved to be bounded
from below by $\theta_{31}^{-1}
= 1.08544\ldots$ where
$\theta_{n}$,
$ n \geq 2$, is the unique root in $(0,1)$ of the
trinomial $-1+x+x^n$. Lehmer's number
$1.176280\ldots$ belongs to the interval
$(\theta_{12}^{-1}, \theta_{11}^{-1})$.
We conjecture that there is no Salem number in
$(\theta_{31}^{-1}, \theta_{12}^{-1})
=
(1.08544\ldots, 1.17295\ldots)$.
For proving the Main Theorem, 
the algebraic and analytic properties of the dynamical 
zeta function of the R\'enyi-Parry numeration system 
are used, with real bases running over the set of real
reciprocal algebraic integers, and variable 
tending to 1.
\end{abstract}

\medskip

\bigskip

\vspace{3cm}

\maketitle


\vspace{0.5cm}


\vspace{0.6cm}

\begin{center}
- Dedicated to the 75th birthday of Christiane Frougny -
\end{center}

\vspace{2cm}

\tableofcontents

\newpage

\section{Introduction}
\label{S1}

A {\em Salem number} is
an algebraic integer $\beta > 1$ such that its Galois conjugates
$\beta^{(i)}$ satisfy:
$|\beta^{(i)}| \leq 1$ for all $i=1, 2, \ldots, m-1$,
with $m = {\rm deg}(\beta) \geq 4$,
$\beta^{(0)} = \beta$ and at least one conjugate 
$\beta^{(i)}, i \neq 0$,
on the unit circle
\cite{bertin} \cite{bertinpathiauxdelefosse}
\cite{bertinetal}.  
All the 
Galois conjugates 
of a Salem number $\beta$ lie on the unit circle, 
by pairs of complex conjugates, except
$1/\beta$ which lies in the open interval $(0,1)$.
Salem numbers are of even degree $\geq 4$.
The minimal polynomial 
$P_{\beta}(X) \in \zb[X]$ 
of $\beta$ is reciprocal. 
Recall that an integer reciprocal
$P(X) \in \zb[X]$ is reciprocal
if it is equal to 
its reciprocal polynomial $P^{*}(X)$, where
$P^{*}(X)$ is
defined by
$P^{*}(X):=X^{\deg(P)} P(1/X)$.
Let us denote by ${\rm T}$
the set of Salem numbers.

\medskip 

\underline{{\bf Q.}} A long-standing basic 
question \cite{bertinetal}
\cite{mackeesmyth} \cite{smyth} \cite{smyth5} is 
about the existence of 
a non-trivial minorant of the set T. 
More precisely,
does there exist a constant $c> 0$ such that:
$\beta \in {\rm T}~ \Longrightarrow ~\beta \geq 1+c$?
This question asks the more general
question of the topology
of T and of its adherence
$\overline{{\rm T}}$: 
in this context, is 1 a limit point of T or not?

\medskip
 
Salem numbers appear in many domains of mathematics, not only in numeration systems or number theory.
For studying the localization of Salem numbers
in $(1,+\infty)$,
the set T has been compared with the 
set S of Pisot numbers by the ``Construction of Salem", introduced in \cite{salem},
and also using association equations,
and the theory of interlacing roots. 
Let us recall the definitions.
A {\it Perron number} is either $1$ or 
a real algebraic integer $\theta > 1$
such that the Galois conjugates
$\theta^{(i)}, i \neq 0$, 
of $\theta^{(0)} := \theta$ satisfy: 
$|\theta^{(i)}| < \theta$. 
Denote by $\mathbb{P}$ the set of Perron numbers. 
A {\it Pisot number} is a Perron number $> 1$ for which 
$|\theta^{(i)}| < 1$ for all $i \neq 0$.
The set of Pisot numbers admits the minorant
$\Theta= 1.3247\ldots$, unique root $> 1$ of
$X^3 -X-1$ by a result of Siegel \cite{siegel}.

Salem (1944) \cite{salem} (Samet \cite{samet})
proved that every Salem number 
is the quotient of two Pisot numbers.
Apart from this direct result, few relations 
are known between Salem numbers and Pisot numbers.
The set of 
Pisot numbers is better known than 
the set of Salem numbers.

To study Salem numbers
association equations between Pisot numbers
and Salem numbers
have been introduced by Boyd 
(\cite{boyd2} Theorem 4.1), Bertin and 
Pathiaux-Delefosse
(\cite{bertinboyd}, 
\cite{bertinpathiauxdelefosse} pp 37-46, 
\cite{bertinetal} chapter 6). 
Association equations 
are generically written
\begin{equation}
\label{assoequation}
(X^2 + 1) P_{Salem} = X P_{Pisot} (X) + 
P^{*}_{Pisot} (X),
\end{equation}
to investigate the links between
infinite collections of Pisot numbers and a 
given Salem number, of respective minimal 
polynomials $P_{Pisot}$ and 
$P_{Salem}$.

The theory of interlacing of roots
on the unit circle
is a powerful tool
for studying classes of polynomials
having special geometry of zeroes 
of modulus one
\cite{lakatos2}
\cite{lakatos4}
\cite{lakatos5} 
\cite{lakatoslosonczi}
\cite{lakatoslosonczi2},
in particular 
Salem polynomials
\cite{mackeesmyth}
\cite{mackeesmyth2}
\cite{mackeesmyth3}.
In \cite{bertinboyd} 
\cite{bertinpathiauxdelefosse}
Bertin and Boyd obtained
two interlacing theorems, namely
Theorem A and Theorem B, turned out to
be fruitful with their limit-interlacing 
versions. McKee
and Smyth in
\cite{mackeesmyth3} obtained new interlacing theorems. 
Theorem 5.3 in \cite{mackeesmyth3} shows 
that all Pisot
numbers are produced by a suitable 
interlacing condition, supporting
the second Conjecture of Boyd,
i.e.\,\eqref{boydconjecture2}. 
Similarly Theorem 7.3 in 
\cite{mackeesmyth3}, using
Boyd's association Theorems, 
shows that all Salem numbers are produced
by interlacing and that a classification of Salem numbers can be made.

In \cite{gvg}
is reconsidered the 
interest of the interlacing Theorems of \cite{bertinboyd},
as potential tools for the study of
families of algebraic integers
in neighbourhoods of 
Salem numbers, as
analogues of those of McKee and Smyth. 
Focusing on Theorem A of \cite{bertinboyd}
association equations are obtained
between Salem polynomials (and/or cyclotomic polynomials) and expansive polynomials,
generically
\begin{equation}
\label{assosalemexpansive}
(z - 1)P_{Salem} (z) = z P_{expansive} (z) -
P^{*}_{expansive} (z),
\end{equation}
which allow to deduce rational 
$n$-dimensional representations
of the neighbourhoods of a Salem number of degree $n$, using the formalism of
Stieltj\`es continued fractions. 
These representations are tools
to study the limit points of sequences
of algebraic numbers in the neighbourhood 
of a given Salem number, the existence of 
Salem numbers
in small neighbourhoods of Salem numbers.

In the same direction
association equations
between Salem numbers and 
(generalized) Garsia numbers
are obtained
by Hare and Panju \cite{harepanju}
using the theory of
interlacing
on the unit circle.

As a counterpart, in number fields, 
Salem numbers are linked
to units: they are given by closed formulas
from Stark units in Chinburg
\cite{chinburg} \cite{chinburg2},
exceptional units in Silverman
\cite{silverman3}. 
From \cite{chinburg2} they are related 
to relative regulators of number fields
\cite{christopoulosmackee}
\cite{costafriedman}
\cite{ghatehironaka}.

In some domains negative Salem numbers
naturally occur.
A negative Salem number is by
definition the opposite of a Salem number,
a negative Pisot number is by
definition the opposite of a Pisot number.
Negative 
Salem numbers occur, e.g.
for graphs or integer
symmetric matrices in 
\cite{mackeesmyth}
\cite{mackeesmyth2}
\cite{mackeesmyth3},
and in other domains,
like Alexander polynomials of links
of the variable ``$-x$", e.g.
in a Theorem 
of Hironaka \cite{hironaka}.

The topology of the set T is 
certainly related to the set S of Pisot numbers 
by the two Conjectures of Boyd:
\begin{equation}
\label{boydconjecture1}
{\rm Conjecture:}
\qquad S \cup {\rm T}~~ \mbox{is closed}
\end{equation} 
and that the first derived set of 
$S \,\, \cup$ T 
satisfies
\begin{equation}
\label{boydconjecture2}
{\rm Conjecture:}
\qquad S = (S \cup {\rm T})^{(1)}
\end{equation}
Only a part of them is proved.
Indeed, Salem \cite{salem2} proved that
the set S of Pisot numbers is closed, and 
${\rm S} \subset \overline{{\rm T}}$.
Its successive derived sets
${\rm S}^{(i)}$,
were extensively studied
by Dufresnoy and Pisot 
\cite{dufresnoypisot},
and their students (Amara \cite{amara}, \ldots),
by means of compact families of 
meromorphic functions, following ideas of Schur.
This analytic approach is reported
extensively
in the book \cite{bertinetal}. 
Conjecture \eqref{boydconjecture2}, if true, 
would imply that
all Salem numbers $< \Theta = 1.3247\ldots$, 
would be
isolated. 
The smallest Salem number known
is Lehmer's number
\begin{equation}
\label{lehmersnumberitself}
\tau := 1.17628\ldots,
\end{equation}
discovered by Lehmer in 1933 \cite{lehmer}, 
with
minimal polynomial (named ``Lehmer's polynomial"):
\begin{equation}
\label{lehmerpolynomialdefinition}
X^{10}+X^9 -X^7 -X^6 -X^5 -X^4 -X^3 + X + 1.
\end{equation}
Lehmer's number was proved to be isolated in
\cite{smyth}.
In \cite{lehmer} Lehmer discovered other small 
Salem numbers. 
There are now part of the list of
Mossinghoff which contains the smallest
Mahler measures of algebraic integers
\cite{mossinghoff2} known.
For degrees up to 180,
the list of Mossinghoff \cite{mossinghofflist}  
(2001),
with
contributions of Boyd, Flammang, 
Grandcolas, Lisonek,
Poulet, Rhin and Sac-Ep\'ee, Smyth,
gives primitive, irreducible, noncyclotomic
integer polynomials of degree at most 180 and of Mahler measure less than 1.3
\cite{mossinghoffrhinwu}; 
this list is complete for degrees less than 40, and, for {\bf Salem numbers}, 
contains the list of the 
47 known smallest Salem numbers, all of degree
$\leq 44$.

\bigskip

The present note is a contribution of the domain of 
numeration systems to these
studies, with techniques
of numeration in variable base. 
It is the proof that the Conjecture of Lehmer,
restricted to Salem numbers, is true, i.e. it is a positive answer to the question \underline{{\bf Q.}}.
For every $n \geq 5$, let $\theta_{n}$ be the 
unique root of the trinomial
$-1+x+x^n$ in $(0,1)$;
its inverse $\theta_{n}^{-1} > 1$ is a Perron number.
The sequence  $(\theta_{n}^{-1})_{n \geq 5}$ 
is strictly decreasing, 
and satisfies: 
$\lim_{n \to \infty} \theta_{n}^{-1} = 1$.
We report the reader
to \cite{vergergaugry6} for details on the family of trinomials
$(-1+x+x^n)_{n \geq 5}$. 
In this note we prove
the following result.

\vspace{0.2cm}

\begin{theorem}[ex-Lehmer conjecture for Salem numbers]
\label{mainpetitSALEM}
The set {\rm T} is bounded from below:
$$\beta \in {\rm T}
\qquad
\Longrightarrow
\qquad
\beta > 
\theta_{31}^{-1} = 1.08544\ldots$$
\end{theorem}

\vspace{0.1cm}

Lehmer's number $1.17628\ldots$ 
belongs to the interval 
$(\theta_{12}^{-1}, \theta_{11}^{-1})$ (cf Table 1).
This interval does not contain any other 
known Salem number. If there is another one,
its degree should be greater than 180. 
After many attempts 
(by Denis Dutykh and the author), and a compilation of the
literature \cite{vergergaugryPANO} on small Salem numbers,
to find 
Salem numbers smaller than Lehmer's number,
we formulate:

\begin{conjecture}
There is no Salem number in the interval
$$(\theta_{31}^{-1}, \theta_{12}^{-1})
= (1.08544\ldots, 1.17295\ldots).$$
\end{conjecture}

The proof of Theorem \ref{mainpetitSALEM} 
(in Section \ref{S4})
uses the analytic function which is the dynamical zeta function,
$\zeta_{\beta}(z)$, of the 
$\beta$-transformation, for $\beta > 1$ running over
the set of reciprocal algebraic integers.
In Section \ref{S2} the basic properties
of the function $\zeta_{\beta}(z)$ 
and the R\'enyi-Parry 
numeration dynamical system are recalled.
Indeed, if 
$1 < \beta  < \theta_{31}^{-1}$, 
this function systematically admits a lenticular pole
which is non-real, of modulus $< 1$, in an
angular sector of the open unit disk
containing 1; 
in Subsection 
\ref{S4.1} the existence of such a pole is proved.
In Subsection \ref{S4.2} {\em rewriting trails},
in the numeration system in base $\beta$,
are developped to
allow to pass from 
$\beta$-representations of 1 from 
$\zeta_{\beta}(z)$ 
to $\beta$-representations of 1
deduced the minimal polynomial $P_{\beta}$
of $\beta$.
Applying Kala-Vavra's Theorem
in Subsection \ref{S4.3} implies
the identification of this lenticular pole
as a Galois conjugate of $1/\beta$, then of $\beta$.
But the contradiction appears if we assume
in particular
that $\beta$ is a Salem number
which is $< \theta_{31}^{-1}$, 
since a Salem number 
$\beta$ never
admits a non-real conjugate in the open unit disk.
The only conjugate of $\beta$ in the open unit disk is
real and is $1/\beta$.
 
Section \ref{S3} gathers the R\'enyi $\beta$-expansions 
$d_{\beta}(1)$ of unity of the small Salem numbers 
known, those which can found in 
Lehmer's paper \cite{lehmer}
or those from the list of Mossinghoff
\cite{mossinghofflist}. Inthere it is shown 
how to compute readily $\zeta_{\beta}(z)$ from
$d_{\beta}(1)$ for such Salem numbers $\beta$
in the interval $(\tau=1.176280\ldots, \Theta
= \theta_{5}^{-1} = 1.32\ldots)$.
This provides examples of Salem numbers
close to, and slightly greater than, 
Lehmer's number $\tau$
while the proof 
of Theorem \ref{mainpetitSALEM} in Section \ref{S4}
is concerned with hypothetical
Salem numbers $< \tau$.

\vspace{0.1cm}

\section{Dynamical zeta function of the R\'enyi-Parry numeration dynamical system}
\label{S2}

Let $\beta$ be a real number, $1 < \beta < 2$.
Denote $\mathcal{A} := \{0,1\}$.
We refer the reader to Lothaire
\cite{lothaire}, Chap. 7 written 
by Christiane Frougny.
Let us fix the notations.
A representation in base $\beta$ 
(or $\beta$-representation) of a real number $x > 0$ 
is an infinite word
$(x_i)_{i \geq 1}$ of
$\mathcal{A}^{\nb}$
such that
$$x = \sum_{i \geq 1} \, x_i \beta^{-i} \, .$$
A particular $\beta$-representation,
called the {\em $\beta$-expansion}, 
or the greedy $\beta$-expansion,
and denoted
by $d_{\beta}(x)$, of $x$
can be computed either by the 
greedy algorithm, or equivalently by the
{\em $\beta$-transformation}
$$T_{\beta} : \,\, x \, \mapsto \,  \beta x \quad(\hspace{-0.3cm}\mod 1) 
= \{\beta x \}.$$
The dynamical system 
$([0,1], T_{\beta})$ is called the
R\'enyi-Parry numeration system in base $\beta$
\cite{parry} \cite{renyi},
the iterates of $T_{\beta}$ providing
the successive digits $x_i$ of $d_{\beta}(x)$
\cite{liyorke}.
Denoting $T_{\beta}^{0} := {\rm Id}, 
T_{\beta}^{1} := T_{\beta}, 
T_{\beta}^{i} := T_{\beta} (T_{\beta}^{i-1})$
for all $i \geq 1$, we have:
$$d_{\beta}(x) = (x_i)_{i \geq 1}
\qquad \mbox{{\rm if and only if}} \qquad
x_ i = \lfloor \beta T_{\beta}^{i-1}(x) \rfloor$$
and we write the $\beta$-expansion of $x$ as
\begin{equation}
\label{xexpansion}
x \, = \, \cdot x_1 x_2 x_3 \ldots 
\qquad \mbox{instead of}
\qquad x = \frac{x_1}{\beta} +\frac{x_2}{\beta^2} +
\frac{x_3}{\beta^3} +
\ldots .
\end{equation}
The digits $x_i$ 
depend upon $\beta$. In particular,
the $\beta$-expansion of $1$ is by 
definition
denoted by
\begin{equation}
\label{renyidef}
d_{\beta}(1) = 0.t_1 t_2 t_3 \ldots
\qquad
{\rm and ~uniquely ~corresponds~ to}
\qquad 1 = \sum_{i=1}^{+\infty} t_i \beta^{-i}\, ,
\end{equation}
where 
\begin{equation}
\label{digits__ti}
t_1 = \lfloor \beta \rfloor,
t_2 = \lfloor \beta \{\beta\}\rfloor = \lfloor \beta T_{\beta}(1)\rfloor,
t_3 = \lfloor \beta \{\beta \{\beta\}\}\rfloor = \lfloor \beta T_{\beta}^{2}(1)\rfloor,
\ldots
\end{equation} 
Denote by $<_{lex}$ the lexicographical
ordering relation on $\{0,1\}^{\nb}$.
The $\beta$-expansion of $1$ plays an important part
in the following, by the fact that 
the map $\beta \to d_{\beta}(1)$
preserves the ordering as follows.
\begin{proposition}
\label{variationbasebeta}
Let $\alpha, \beta  \in (1,2)$ with 
$\alpha \neq \beta$. 
If the R\'enyi $\alpha$-expansion of 1 is
$$d_{\alpha}(1) = 0. t'_1 t'_2 t'_3\ldots, 
\qquad ~i.e.
\quad
1 ~=~ \frac{t'_1}{\alpha} + \frac{t'_2}{\alpha^2} + \frac{t'_3}{\alpha^3} + \ldots$$
and the R\'enyi $\beta$-expansion of 1 is
$$d_{\beta}(1) = 0. t_1 t_2 t_3\ldots, 
\qquad ~i.e. 
\quad
1 ~=~ \frac{t_1}{\beta} + \frac{t_2}{\beta^2} + \frac{t_3}{\beta^3} + \ldots,$$
then $\alpha < \beta$ if and only if $(t'_1, t'_2, t'_3, \ldots) 
<_{lex} (t_1, t_2, t_3, \ldots)$. 
\end{proposition}

\begin{proof}
Lemma 3 in Parry \cite{parry}.
\end{proof}

Since the sequence 
$(\theta_{n}^{-1})_{n \geq 2}$ 
is (strictly) decreasing and tends to $1$ if
$n$ tends to infinity
\cite{vergergaugry6}, it induces
the partitioning 
\begin{equation}
\label{jalonnement}
\bigl( 1, \frac{1+\sqrt{5}}{2}\, \bigr] ~=~
\bigcup_{n=2}^{\infty}
\left[ \, \theta_{n+1}^{-1} , \theta_{n}^{-1}
\, 
\right)
~~ \bigcup ~~\left\{
\theta_{2}^{-1}
\right\}.
\end{equation}

\begin{definition}
Let $\beta \in ( 1, \frac{1+\sqrt{5}}{2}\, ]$ be a real number.
The integer $n \geq 3$ such that
$\theta_{n}^{-1} \leq \beta < \theta_{n-1}^{-1}$
is called the {\it dynamical degree} 
of $\beta$, and
is denoted by ${\rm dyg}(\beta)$.
By convention we put:
${\rm dyg}(\frac{1+\sqrt{5}}{2}) = 2$. 
\end{definition}
The function $\beta \to n={\rm dyg}(\beta)$ 
is locally
constant on the interval 
$(1, \frac{1+\sqrt{5}}{2}]$, 
takes all values in $\mathbb{N} \setminus \{0,1\}$,
and satisfies:
$\lim_{\beta > 1, \beta \to 1} \dyg(\beta) = +\infty$.
We have:
$$\beta ~~\mbox{tends to 1}
\qquad \mbox{if and only if}\qquad \quad
n=\dyg(\beta)~~\mbox{tends to} +\infty.$$
\begin{definition}
\label{parryupperfunction}
Let $\beta \in (1, (1+\sqrt{5})/2]$ be a real
number, and $d_{\beta}(1) = 0. t_1 t_2 t_3 \ldots$ its
R\'enyi $\beta$-expansion of 1.
The power series $f_{\beta}(z) :=
-1 + \sum_{i \geq 1} t_i z^i$
of the complex variable $z$
is called the {\it Parry Upper function} 
at $\beta$.
\end{definition}



From 
$(t_i)_{i \geq 1} \in \mathcal{A}^{\nb}$
is built  
$(c_i)_{i \geq 1} \in \mathcal{A}^{\nb}$,
defined by
$$
c_1 c_2 c_3 \ldots := \left\{
\begin{array}{ll}
t_1 t_2 t_3 \ldots & \quad \mbox{if ~$d_{\beta}(1) = 0.t_1 t_2 \ldots$~ is infinite,}\\
(t_1 t_2 \ldots t_{q-1} (t_q - 1))^{\omega}
& \quad \mbox{if ~$d_{\beta}(1)$~ is finite,
~$= 0. t_1 t_2 \ldots t_q$,}
\end{array}
\right.
$$
where $( \, )^{\omega}$ means that the word within $(\, )$ is indefinitely repeated.
The sequence $(c_i)_{i \geq 1}$ is the unique
element of $\mathcal{A}^{\nb}$
which allows to obtain all the admissible
$\beta$-expansions of all the elements of
$[0,1)$.

\begin{definition}[Conditions of Parry]
A sequence $(y_i)_{i \geq 0}$ of elements of
$\mathcal{A}$ (finite or not) is said
{\it admissible} if 
\begin{equation}
\label{conditionsParry}
\sigma^{j}(y_0, y_{1}, y_{2}, \ldots) :=
(y_j, y_{j+1}, y_{j+2}, \ldots) <_{lex}
~(c_1, \, c_2, \, c_3, \,  \ldots) 
\quad \mbox{for all}~ j \geq 0.
\end{equation}
\end{definition}
The operator $\sigma$ on 
$\mathcal{A}^{\nb}$ is the one-sided shift.
\begin{definition}
A sequence 
$(a_i)_{i \geq 0} \in \mathcal{A}^{\nb}$
satisfying  \eqref{lyndonEQ} is said to be
{\it Lyndon (or self-admissible)}:
\begin{equation}
\label{lyndonEQ}
\sigma^{n}(a_0, a_{1}, a_{2}, \ldots)
=
(a_n, a_{n+1}, a_{n+2}, \ldots) <_{lex} (a_0, a_1, a_2, \ldots) 
\qquad \mbox{for all}~ n \geq 1.
\end{equation} 
\end{definition}
The terminology comes from the introduction
of such words by Lyndon in \cite{lyndon}, in honour of his work.

\begin{theorem}
\label{zeronzeron}
Let $n \geq 2$. A real number 
$\beta \in ( 1, \frac{1+\sqrt{5}}{2}\, ]$ 
belongs to   
$[\theta_{n+1}^{-1} , \theta_{n}^{-1})$ if and only if the 
R\'enyi $\beta$-expansion of unity is of the form
\begin{equation}
\label{dbeta1nnn}
d_{\beta}(1) = 0.1 0^{n-1} 1 0^{n_1} 1 0^{n_2} 1 0^{n_3} \ldots,
\end{equation}
with $n_k \geq n-1$ for all $k \geq 1$.  
\end{theorem}

\begin{proof} 
Since $d_{\theta_{n+1}^{-1}}(1) = 0.1 0^{n-1} 1$ and
$d_{\theta_{n}^{-1}}(1) = 0.1 0^{n-2} 1$, 
Proposition \ref{variationbasebeta} implies that
the condition is sufficient. It is also necessary:
$d_{\beta}(1)$ begins as $0.1 0^{n-1} 1$
for all $\beta$ such that
$\theta_{n+1}^{-1} \leq \beta < \theta_{n}^{-1}$.
For such $\beta$s 
we write $d_{\beta}(1) = 0.1 0^{n-1} 1 u$~  
with digits in the alphabet
$\mathcal{A}_{\beta}
=\{0, 1\}$ common to all $\beta$s, that is
$$u= 1^{h_0} 0^{n_1} 1^{h_1} 0^{n_2} 1^{h_2} \ldots$$
and $h_0, n_1, h_1, n_2, h_2, \ldots$ integers $\geq 0$.
The self-admissibility lexicographic condition
\eqref{lyndonEQ}
applied to the sequence
$(1, 0^{n-1}, 1^{1+h_0}, 0^{n_1}, 1^{h_1}, 0^{n_2}, 1^{h_3}, \ldots)$,
which characterizes uniquely the base of numeration $\beta$, 
readily implies
$h_0 = 0$ and 
$~h_k = 1$ and 
$n_k \geq n-1$ for all $k \geq 1$.
\end{proof}

\begin{definition}
\label{selfadmissiblepowerseries}
A power series $\sum_{j=0}^{+\infty} a_j z^j$,
with $a_j \in \{0, 1\}$ for all $j \geq 0$,
$z$ the complex variable,
is said to be {\it Lyndon (or self-admissible)}
if its
coefficient vector $(a_i)_{i \geq 0}$ 
is Lyndon.
\end{definition}

In Fredholm theory, the Parry Upper function
at $\beta$ is the generalized Fredholm determinant 
(up to the sign) 
of the transfer 
operator of the $\beta$-transformation
(cf Ruelle
\cite{ruelle} \cite{ruelle2} \cite{ruelle3} \cite{ruelle4}
\cite{ruelle5} \cite{ruelle6}
\cite{ruelle7} \cite{ruelle8}
\cite{ruelle9}
and more recently 
Baladi \cite{baladi} \cite{baladi2}
\cite{baladi3},
Baladi and Keller \cite{baladikeller},
Hofbauer \cite{hofbauer},
Hofbauer and Keller \cite{hofbauerkeller},
Milnor and Thurston \cite{milnorthurston},
Parry and Pollicott \cite{parrypollicott},
Pollicott \cite{pollicott} \cite{pollicott2},
Takahashi \cite{takahashi3}
\cite{takahashi4}).

For $n \geq 2$, denote: $G_{n}(X):=-1+X+X^n$.

\begin{proposition}
\label{fbetainfinie}
For $1 < \beta <(1+\sqrt{5})/2$ any real
number,
with $d_{\beta}(1) = 0. t_1 t_2 t_3 \ldots$,
the Parry Upper function $f_{\beta}(z)$
is such that $f_{\beta}(1/\beta) = 0$. 
It is such that
$f_{\beta}(z) + 1$ has coefficients 
in the alphabet
$\{0,1\}$ and its coefficient vector is 
Lyndon.
It takes the form
\begin{equation}
\label{fbetalyndon}
f_{\beta}(z) = G_{\dyg(\beta)}(z) + z^{m_1} + 
z^{m_2} + \ldots + z^{m_q} + z^{m_{q+1}} + \ldots
\end{equation}
with~ $m_1 - \dyg(\beta) \geq  \dyg(\beta) -1$,
$m_{q+1} - m_q \geq  \dyg(\beta) -1$ for
$q \geq 1$.
Conversely, given a power series 
\begin{equation}
\label{fbetalyndonconverse}
-1 + z + z^n  + z^{m_1} + 
z^{m_2} + \ldots + z^{m_q} + z^{m_{q+1}} + \ldots
\end{equation}
with $n \geq 3$,
$m_1 - n \geq  n -1$,
$m_{q+1} - m_q \geq  n -1$ for
$q \geq 1$, then there exists an unique
$\beta \in (1, (1+\sqrt{5})/2)$ for which
$n = \dyg(\beta)$ with
$f_{\beta}(z)$ equal to
\eqref{fbetalyndonconverse}.
 
Moreover, if $\beta$,
$1 < \beta <(1+\sqrt{5})/2$,
is a reciprocal
algebraic integer, in particular a Salem number,
the power series
\eqref{fbetalyndon} is never a polynomial.
\end{proposition}

\begin{proof}
The expression of 
$f_{\beta}(z)$ readily comes
from Theorem \ref{zeronzeron}.
Let us prove the last claim.
Assume that $\beta$ is a reciprocal
algebraic integer and that 
$f_{\beta}(z)$ is a polynomial. The
polynomial $f_{\beta}(z)$ would vanish at
the two real zeroes
$\beta$ and $1/\beta$. But the sequence
$-1\, t_1\, t_2\, t_3 \ldots$ has only 
one sign change. 
By Descartes's rule 
we obtain a contradiction.
\end{proof}

\begin{definition}
A real number $\beta > 1$ is said to be
a simple Parry number if $d_{\beta}(1)$ is finite
(i.e. if it ends in infinitely many zeroes). 
It is a Parry number if it is simple or if
$d_{\beta}(1)$ is eventually periodic, with a preperiod
of length $\geq 0$ and a nonzero period. 
\end{definition}
Parry \cite{parry} proved that the set of simple Parry numbers is dense in $(1,+\infty)$. 
Salem numbers are never simple Parry numbers.

In number theory, inequalities are
often associated to collections of half-spaces in
euclidean or adelic Geometry of Numbers
(e.g. Minkowski's Theorem, etc).
The conditions of Parry (at $\beta$)
are of totally different nature
since they refer to 
a reasonable control, order-preserving, 
of the gappiness (lacunarity)
of the coefficient vectors of the
power series which are the Parry Upper functions,
when $\beta$ is close to 1.

\begin{theorem}
\label{dynamicalzetatransferoperator}
Let $\beta \in (1,\theta_{2}^{-1})$. Then,
the Artin-Mazur
dynamical zeta function
$\zeta_{\beta}(z)$ defined 
by 
\begin{equation}
\label{dynamicalfunction}
\zeta_{\beta}(z) := \exp\Bigl(
\sum_{n=1}^{\infty} \, 
\frac{\#\{x \in [0,1] \mid 
T_{\beta}^{n}(x) = x\}}{n} \, z^n
\Bigr),
\end{equation}
counting the number of periodic 
points of period dividing $n$, has a sense,
is nonzero and meromorphic
in $\{z: |z| < 1\}$, and such that
$1/\zeta_{\beta}(z)$ is holomorphic
in $\{z: |z| < 1\}$,
\end{theorem}

\begin{proof}
Theorem 2 in 
\cite{baladikeller}, 
assuming that
the set of intervals 
$([0, a_1),[a_1 , 1])$
forming the partition of
$[0,1]$ is generating;
In \cite{ruelle5} \cite{ruelle8} Ruelle shows that this assumption is not necessary, showing how to remove this obstruction.
\end{proof}

Theorem \ref{dynamicalzetatransferoperator}
has been 
stated in Baladi and Keller
\cite{baladikeller} under more general
assumptions.
Theorem \ref{dynamicalzetatransferoperator} 
had been previously 
conjectured by Hofbauer and Keller
\cite{hofbauerkeller} 
for piecewise monotone maps, for the case
where the function $g$ occuring in the
transfer operators is piecewise constant.
(cf also Mori \cite{mori} \cite{mori2}).
The case $g=1$ in the transfer operators
was studied by Milnor and Thurston \cite{milnorthurston},
Hofbauer \cite{hofbauer}.

The relations between 
the poles of the dynamical
zeta function $\zeta_{\beta}(z)$ and
the zeroes of the Parry Upper function
$f_{\beta}(z)$ 
come from Theorem \ref{dynamicalzetatransferoperator}
and from the following theorem.

\begin{theorem}
\label{parryupperdynamicalzeta}
Let $\beta \in (1,\theta_{2}^{-1})$.
Then the Parry Upper function
$f_{\beta}(z)$ satisfies
\begin{equation}
\label{parryupperdynamicalzeta_i}
(i)\qquad f_{\beta}(z) = - \frac{1}{\zeta_{\beta}(z)} 
\qquad \mbox{if}~ \beta ~\mbox{is not a simple Parry number},
\end{equation}
and
\begin{equation}
\label{parryupperdynamicalzeta_ii}
(ii)\quad
f_{\beta}(z) = - \frac{1 - z^N}{\zeta_{\beta}(z)}
\qquad
\mbox{if $\beta$ is a simple Parry number}
\end{equation}
where $N$, 
which depends upon $\beta$, is the minimal positive integer such that $T_{\beta}^{N}(1) = 0$.
It is holomorphic in the open unit disk
$\{z: |z| < 1\}$.
It has no zero in
$\{z: |z| \leq 1/\beta\}$ except $z=1/\beta$
which is a simple zero.
\end{theorem}

\begin{proof}
Theorem 2.3 and Appendix A
in Flatto, Lagarias and Poonen 
\cite{flattolagariaspoonen}; 
Theorem 1.2 in
Flatto and Lagarias 
\cite{flattolagarias}, I;
Theorem 3.2
in Lagarias \cite{lagarias}.
From Takahashi \cite{takahashi},
Ito and Takahashi
\cite{itotakahashi}, these authors
deduce 
\begin{equation}
\label{zetafunctionfraction}
\zeta_{\beta}(z) = \frac{1 - z^N}{(1 - \beta z)
\Bigl(
\sum_{n=0}^{\infty}T_{\beta}^{n}(1) \, z^n
\Bigr)}
\end{equation}
where ``$z^N$" has to be replaced by ``$0$"
if $\beta$ is not a simple Parry number.
Since $\beta T_{\beta}^{n}(1) =
\lfloor \beta T_{\beta}^{n}(1)
\rfloor +
\{\beta T_{\beta}^{n}(1)\} = 
t_{n+1} + T_{\beta}^{n+1}(1)$
by \eqref{digits__ti}, for $n \geq 1$,
expanding the power series of the denominator
\eqref{zetafunctionfraction} readily gives:
\begin{equation}
\label{fbetazetabeta}
-1 +t_1 z + t_2 z^2 +\ldots
= f_{\beta}(z) 
= 
-(1 - \beta z)
\Bigl(
\sum_{n=0}^{\infty}T_{\beta}^{n}(1) \, z^n
\Bigr).
\end{equation}
The zeroes of smallest modulus are
characterized in Lemma 5.2, Lemma 5.3 and Lemma 5.4
in \cite{flattolagariaspoonen}. 
\end{proof}

Let $\beta \in (1,\theta_{2}^{-1})$.
Since the power series $f_{\beta}(z)$ has coefficients
in the finite set
$\mathcal {A} \cup \{-1\}$, its radius of convergence
is $\geq 1$ by Hadamard's formula, and it obeys
the Carlson-Polya dichotomy
(Bell and Chen \cite{bellchen},
Bell, Miles and Ward
\cite{bellmilesward},
Carlson \cite{carlson}, 
P\'olya \cite{polya}, 
Szeg\H{o} \cite{szego}).
Applying the Carlson-Polya dichotomy
gives the following
equivalence.

\begin{theorem}
\label{carlsonpolya}
Let $\beta \in (1,\theta_{2}^{-1})$.
The real number $\beta$
is a Parry number if and only if
the Parry Upper function
$f_{\beta}(z)$ is a rational function, 
equivalently if and only
if $\zeta_{\beta}(z)$ is a rational function.

Moreover, the set of 
nonParry numbers 
in $(1, \infty)$ is not empty.
If $\beta$ is not a Parry number, 
then the unit circle
$\{z: |z| = 1\}$ is the natural boundary
of  
$f_{\beta}(z)$, equivalently
of $\zeta_{\beta}(z)$. 
\end{theorem}

\begin{proof}
\cite{vergergaugry6}.
The set of nonParry numbers $\beta$
in $(1, \infty)$ is not empty
as a consequence of Fekete-Szeg\H{o}'s Theorem 
\cite{feketeszego} 
since the radius of convergence
of $f_{\beta}(z)$ is equal to 1 in any case,
and that its domain of definition always contains
the open unit disk which has a transfinite
diameter equal to 1.
Since the set of Parry numbers $\beta$
in $(1, \infty)$ is nonempty,
the dichotomy between the set of
Parry numbers 
and the set of nonParry numbers, both non empty, 
in $(1,\infty)$ has a sense.
\end{proof}

Recall that the set of simple Parry numbers is dense \cite{parry} in $(1,+\infty)$ and that
the set of Parry numbers which are not simple
is infinite since it contains intinitely many 
Pisot numbers by Schmidt \cite{schmidt}, 
Bertrand-Mathis \cite{bertrandmathis}.

\vspace{0.1cm}

\section{Dynamics of the small Salem numbers known}
\label{S3}

The set $\pb$ of Perron numbers is dense
in $[1,+\infty)$. It
contains the subset $\pb_P$ of the
Parry numbers by a result of Lind 
\cite{lind} (Blanchard \cite{blanchard},
Boyle \cite{boyle}, 
Denker, Grillenberger and Sigmund 
\cite{denkergrillenbergersigmund}, 
Frougny in 
\cite{lothaire} chap.7). 
The set $\pb \setminus \pb_P$
is not empty (Akiyama \cite{akiyama}, and as a
consequence of Fekete-Szeg\H{o}'s Theorem 
\cite{feketeszego});
it would contain 
all Salem numbers of large degrees,
by Thurston \cite{thurston2} p. 11. 
Parry (\cite{parry}, Theorem 5) proved that
the subcollection 
of simple Parry numbers is dense
in $[1, \infty)$.

\medskip

In the opposite direction 
a Conjecture of K. Schmidt \cite{schmidt}
asserts that Salem numbers are all Parry 
numbers. 
For Salem numbers $\beta$
of degree $\geq 6$,
Boyd 
\cite{boyd15} established
a simple probabilistic model, based on 
the frequencies of digits
occurring in the R\'enyi $\beta$-expansions of unity,
to conjecture that, more realistically,
Salem numbers are dispatched into 
the two sets of Parry numbers and nonParry numbers,
each of them with densities $> 0$. 
This model, coherent with Thurston's one 
(\cite{thurston2}, p. 11),
is in contradiction with the conjecture of 
K. Schmidt.

\medskip

This dichotomy of Salem numbers
was verified
by Hichri \cite{hichri} 
\cite{hichri2} \cite{hichri3} 
for Salem numbers of 
degree 8. Hichri 
further developped the heuristic approach of Boyd
for Salem numbers of degree 8. 
The Salem numbers of degree $\leq 8$ are all greater
than $1.280638\ldots$ from 
\cite{mossinghofflist}.
Salem numbers of degree 4
are Parry numbers \cite{boyd13} \cite{boyd14}.
Boyd's model covers the set of
Salem numbers smaller than 
Lehmer's number, if any.

The small Salem numbers found 
by Lehmer in \cite{lehmer},
reported in Table 2,
either given 
by their minimal polynomial or 
equivalently by their $\beta$-expansion,
are Parry numbers whose 
R\'enyi $\beta$-expansion of unity
has a preperiod length equal to 1.
Table 1 gives the subcollection of those
Salem numbers $\beta$ which are Parry numbers, of
degree $\leq 44$,
within the intervals of extremities
the Perron numbers 
$\theta_{n}^{-1}, \,n = 5, 6, \ldots, 12$.
In each interval 
$(\theta_{n}^{-1}, \theta_{n-1}^{-1})$ the 
Salem numbers $\beta$ have 
the same dynamical degree $\dyg(\beta)$,
while a certain disparity
of the degrees $\deg(\beta)$ occurs.
The remaining Salem numbers in \cite{mossinghofflist}
are very probably nonParry numbers 
though proofs are not available yet; they are not included
in Table 1. Apart from them,
the other Salem numbers which exist in the intervals
$(\theta_{n}^{-1}, \theta_{n-1}^{-1}), 
\,n \geq 6$,
if any, should be of degrees
$ > 180$.

{\small
\begin{center}
\begin{tabular}{c|c|c|c|c}
$\dyg$ & $\deg$ & $\beta$ & $P_{\beta,P}$ & $d_{\beta}(1)$\\
\hline
\\
5 & 3 & $\theta_{5}^{-1}=1.324717$ & 5 & \qquad 
\qquad $0.1 0^{3} 1$
\quad \mbox{smallest Pisot number - Siegel}\\
\hline
\\
6 & 18 & $1.29567$ & 22 & $0.1(0^4 1 0^9 1 0^6)^{\omega}$\\
6 & 10 & $1.293485$ & 12 & $0.1(0^4 1 0^6)^{\omega}$\\
6 & 24 & $1.291741$ & 24 irr. & $0.1(0^4 1 0^{11} 1 0^6)^{\omega}$ \\
6 & 26 & $1.286730$ & 30 & $0.1(0^4 1 0^{17} 1 0^6)^{\omega}$ \\
6 & 34 & $1.285409$ & 38 & $0.1(0^4 1 0^{25} 1 0^6)^{\omega}$ \\
6 & 30 & $1.285235$ & 45 & $0.1(0^4 1 0^{32} 1 0^6)^{\omega}$ \\
6 & 44 & $1.285199$ & 66 & $0.1(0^4 1 0^{54} 1 0^6)^{\omega}$ \\
6 & 6 & $\theta_{6}^{-1}=1.285199$ & 6 irr.& 
\qquad $0.1 0^{4} 1$
\qquad \qquad Perron number\\
\hline
\\
7 & 26 & $1.285196$ & 44 & $
0.1(0^5 1 0^5 1 0^5 1 0^5 1 0^5 1 0^{5} 1 0^7)^{\omega}$ \\
7 & 26 & $1.281691$ & .. & $
0.1(0^5 1 0^5 1 0^9 1 0^5 1 0^{17} 1 0^{7} 1 0^6 1 0^6 1 0^7 1 0^{12})^{\omega}$ \\
7 & 8 & $1.280638$ & 20 & $
0.1(0^5 1 0^5 1 0^7)^{\omega}$ \\
7 & 10 & $1.261230$ & 14 & $
0.1(0^5 1 0^7)^{\omega}$ \\
7 & 24 & $1.260103$ & 28 & $
0.1(0^5 1 0^{13} 1 0^7)^{\omega}$ \\
7 & 18 & $1.256221$ & 36 & $
0.1(0^5 1 0^{21} 1 0^7)^{\omega}$ \\
7 & 7 & $\theta_{7}^{-1}=1.255422$ & 7 irr.& 
\qquad \quad $0.1 0^{5} 1$
\qquad \qquad Perron number\\
\hline\\
8 & 18 & $1.252775$ & 120 & $
0.1(0^6 1 0^{6} 1 0^{10} 1 0^{16} 1 0^{12} 1 0^7 1 0^{12} 0^{16} 1 0^{10}1 0^6 1 0^8)^{\omega}$ \\
8 & 12 & $1.240726$ & 48 & $
0.1(0^6 1 0^{11} 1 0^7 1 0^{11} 1 0^8)^{\omega}$ \\
8 & 20 & $1.232613$ & 41 & $
0.1(0^6 1 0^{24} 1 0^8)^{\omega}$ \\
8 & 8 & $\theta_{8}^{-1}=1.232054$ & 8 irr.& 
\qquad $0.1 0^{6} 1$
\qquad \qquad Perron number\\
\hline\\
9 & 10 & $1.216391$ & 18 & $
0.1(0^7 1 0^9)^{\omega}$ \\
9 & 9 & $\theta_{9}^{-1}=1.213149$ & 9 irr.& 
\qquad $0.1 0^{7} 1$
\qquad \qquad Perron number\\
\hline\\
10 & 14 & $1.200026$ & 20 & $
0.1(0^8 1 0^{10})^{\omega}$ \\
10 & 10 & $\theta_{10}^{-1}=1.197491$ & 10 irr.& 
\qquad $0.1 0^{8} 1$ \qquad \qquad Perron number\\
\hline\\
11 & 9 & $\theta_{11}^{-1}=1.184276$ & 11 & 
\qquad
$0.1 0^{9} 1$
\qquad \qquad Perron number
\\
\hline\\
12 & 10 & $1.176280$ & 75 & Lehmer's number $:
0.1(0^{10} 1 0^{18} 1 0^{12} 1 0^{18} 1 0^{12})^{\omega}$ \\
12 & 12 & $\theta_{12}^{-1}=1.172950$ & 12 irr.& 
\qquad $0.1 0^{10} 1$
\qquad \qquad Perron number\\
\hline
\end{tabular}

\end{center}
}

\begin{center}
Table 1.~ Smallest Salem numbers 
$\beta < 1.3$ of degree $\leq 44$, which are Parry numbers, computed from the ``list of Mossinghoff" 
\cite{mossinghofflist}
of irreducible monic
integer polynomials of degree $\leq 180$.
 In Column 1 is reported 
the dynamical degree of $\beta$. 
Column 4 gives the degree 
$d_P$ of the Parry polynomial
$P_{\beta,P}$ of $\beta$; $P_{\beta,P}$ is reducible except if ``irr." is mentioned.
\end{center}

The Salem numbers $\beta$,
$\theta_{12}^{-1} < \beta < \theta_{5}^{-1}$, in
Table 1,
of degree $\leq 44$, are Parry numbers.
The Parry polynomial 
$P_{\beta,P}(X)$ of a Salem number $\beta$ is 
determined by the equation
$$f_{\beta}(z) = \frac{- P^{*}_{\beta,P}(z)}{1-z ^{p+1}}  
\qquad
\mbox{if}~ 
d_{\beta}(1) = 0.t_1 t_2 \ldots t_m (t_{m+1}t_{m+2}
\ldots t _{m+p+1})^{\omega} ~\mbox{has period length}
~ p+1 \geq 1.
$$

\vspace{0.2cm}
\noindent
\begin{center}
\begin{tabular}{c|c|c|l}
\label{smallSalemexpansions}
$\deg(\beta)$ & 
$\beta $
&
minimal polynomial of $\beta$
& 
$d_{\beta}(1)$\\
\hline
& & &  \\
4 & 
$1.722\ldots$
&
$X^4 - X^3 - X^2 -X +1$
&
$0 . 1 (1 0 0)^{\omega}$\\
6 & 
$1.401\ldots$
&
$X^6 - X^4 - X^3 -X^2 +1$
&
$0 . 1 ( 0^2 1 0^4)^{\omega}$\\
8 & 
$1.2806\ldots$
&
$X^8 - X^5 - X^4 -X^3 +1$
&
$0 . 1 (0^5 1 0^5 1 0^{7})^{\omega}$\\
10 & 
$1.17628\ldots$
&
$X^{10} + X^9 - X^7 - X^6 - X^5 $
&
$0 . 1 (0^{10} 1 0^{18} 1 0^{12} 1 0^{18} 1 0^{12})^{\omega}$\\
 & &
 $- X^4 -X^3 + X + 1$
 & \\

\end{tabular}
\end{center}
\begin{center}
Table 2.~ The small Salem numbers, with their minimal polynomial,
discovered by Lehmer in \cite{lehmer}.
\end{center}

\vspace{0.3cm}

Using the ``Construction of Salem'',
Hare and Tweedle \cite{haretweedle}
obtain convergent families of Salem numbers, 
all Parry numbers, having as limit points
the limit points of
 the set S of Pisot numbers 
 in the interval $(1, 2)$ (characterized by
 Amara \cite{amara}).
These families of Salem numbers which are Parry numbers 
do not contain Salem numbers
smaller than 
Lehmer's number.

The relations between the digits $(t_i)$
in the R\'enyi $\beta$-expansion of 
unity $d_{\beta}(1)$
of a Salem number $\beta$
and the coefficient vector
of its minimal polynomial are 
obscure in general. However
we have the following general result
on the gappiness, which is never ``extreme". 

\begin{theorem}
\label{lacunarityVG06}
Let $\beta > 1$ be a Salem number.
We assume that $d_{\beta}(1)$ is
gappy in the sense that there exist
two infinite sequences
$\{m_n\}_{n \geq 1}$ 
and $\{s_n\}_{n \geq 0}$
such that
$$1 = s_0 \leq m_1 < s_1 \leq m_2 < s_2 \leq \ldots
\leq m_n < s_n \leq m_{n+1}
< s_{n+1} \leq \ldots$$
with $(s_n - m_n) \geq 2$, $t_{m_n} \neq 0$,
$t_{s_n} \neq 0$
and $t_i = 0$
if
$m_n < i < s_n$ for all $n \geq 1$. 
Then
\begin{equation}
\label{gappiness}
\limsup_{n \to +\infty} \frac{s_n}{m_n} = 1
\end{equation}
\end{theorem}

\begin{proof}
There are two cases: (i) 
$\beta$ is a Parry number: we readily obtain
the result.
(ii) $\beta$ is not a Parry number: then 
${\rm M}(\beta) = \beta$, where M denotes the Mahler measure of an
algebraic number. Then we apply
Theorem 1.1 in
\cite{vergergaugry}, and deduce the result.
Another way is to obtain (ii) as a consequence of 
Theorem 2 in \cite{adamczewskibugeaud}.
\end{proof}

\section{Proof of the Main Theorem \ref{mainpetitSALEM}}
\label{S4} 

For proving Theorem \ref{mainpetitSALEM},
we assume the existence of a Salem number
$\beta$ in an interval
$(\theta_{n}^{-1},\theta_{n-1}^{-1})$, 
$n \geq 32$, and arrive at a contradiction.

\subsection{A non-real lenticular pole of $\zeta_{\beta}(z)$ in a small angular sector}
\label{S4.1}

This Salem number $\beta \in
(\theta_{n}^{-1},\theta_{n-1}^{-1})$, 
$n \geq 32$, would admit 
a R\'enyi $\beta$-expansion of unity
$d_{\beta}(1)$
which is infinite, by Proposition
\ref{fbetainfinie}:
$$d_{\beta}(1) = 
0.1 0^{n-2} 1 0^{n_1} 1 0^{n_2} 1 0^{n_3} \ldots,
$$
with $n_k \geq n-2$ for all $k \geq 1$.
The gappiness is controlled by the integer $n-2$, 
in the sense that any string of zeroes
has length $\geq n-2$, and is never asymptotically
important by Theorem \ref{lacunarityVG06}.
Then
the Parry Upper function $f_{\beta}(z)$ at $\beta$
takes the following form and
is 
characterized by the sequence 
of exponents $(m_q)_{q \geq 0}$:
\begin{equation}
\label{fbet}
f_{\beta}(z) = -1 + z + z^n + z^{m_1} + z^{m_2}
+ z^{m_3} + \ldots = G_{n}(z) +
\sum_{q \geq 1} z^{m_q} ,
\end{equation}
where $m_0 := n$, with the fundamental
minimal gappiness condition:
\begin{equation}
\label{minimalgappiness}
m_{q+1} - m_q \geq n-1 \qquad 
\mbox{for all}~ q \geq 0.
\end{equation}

\medskip

Inside the open unit disk,
the zeroes of $f_{\beta}(z)$ 
are exactly the poles of $\zeta_{\beta}(z)$,
with the same multiplicities,
by Theorem
\ref{parryupperdynamicalzeta}.
For convenience, we will deal with
the zeroes of $f_{\beta}(z)$ instead of the poles of
$\zeta_{\beta}(z)$, in the sequel. 
The terminology
``lenticular poles", in the title of this section,
comes from the fact
that the first zero (cf the definition below)
naturally belongs to a lenticulus of
zeroes, as shown in Figure \ref{example649}, and the 
existence of lenticuli is
the case in general \cite{vergergaugryPROOF}.

\bigskip

The method which will be used
to
detect a non-real zero of
$f_{\beta}(z)$ in the open unit disk
is the method of  Rouch\'e
applied to \eqref{fbet} and to the first zero of
$G_{n}(z)$, in which
$\sum_{q \geq 1} z^{m_q}$ is a perturbation of 
$G_{n}(z) = -1+z+z^n$.
Let us first recall the notations
of the zeroes of the trinomials
$G_n$ and their geometry, from
the factorization of $G_{n}(X) := -1 +X +X^n$ 
by Selmer \cite{selmer} (cf 
\cite{vergergaugry6} Section 2).

\

\subsubsection{Roots of $G_n$ and the asymptotic expansion of the first root}
\label{Gnroots}

Let $n \geq 2$.
Summing in pairs over complex conjugated imaginary roots, the
indexation of the roots and
the factorization of $G_{n}(X)$ 
are taken as follows:
\begin{equation}
\label{factoGG}
G_{n}(X) = (X - \theta_n) \, 
\left(
\prod_{j=1}^{\lfloor \frac{n}{6} \rfloor} (X - z_{j,n}) ( X - \overline{z_{j,n}}) 
\right)
\times q_{n}(X),
\end{equation}
where
$\theta_n$ is the only (real) root 
of $G_{n}(X)$ in the interval $(0, 1)$, where
$$q_{n}(X) ~=~
\left\{
\begin{array}{ll}
\displaystyle
\left(
\prod_{j=1+\lfloor \frac{n}{6} \rfloor}^{\frac{n-2}{2}} (X- z_{j,n}) (X - \overline{z_{j,n}})
\right) \times (X- z_{\frac{n}{2},n})
&
\mbox{if $n$ is even, with}\\
 & \mbox{$z_{\frac{n}{2},n}$ real $< -1$},\\
\displaystyle
\prod_{j=1+\lfloor \frac{n}{6} \rfloor}^{\frac{n-1}{2}} (X- z_{j,n}) (X - \overline{z_{j,n}})
&
\mbox{if $n$ is odd,}
\end{array}
\right.
$$
where the index $j = 1, 2, \ldots$ is such that 
$z_{j,n}$ is a (nonreal) complex zero of $G_{n}(X)$, except if 
$n$ is even and $j=n/2$, such that
the argument $\arg (z_{j,n})$ of $z_{j,n}$
is  $\approx 2 \pi j/n$
and that the family of arguments 
$(\arg(z_{j,n}))_{1 \leq j < \lfloor n/2 \rfloor}$ 
forms a strictly increasing sequence with $j$:
$$0 < \arg(z_{1,n}) < \arg(z_{2,n}) < \ldots 
< \arg(z_{\lfloor \frac{n}{2} \rfloor,n}) \leq \pi.$$
For $n \geq 2$ all the roots of $G_{n}(X)$ 
are simple, and
the roots of $G_{n}^{*}(X) = 1 + X^{n-1} - X^n$, 
as inverses of the roots
of $G_{n}(X)$, are classified in the 
reversed order.

\begin{proposition}
\label{irredGn}
Let $n \geq 2$. 
If $n \not\equiv 5 ~({\rm mod}~ 6)$, 
then $G_{n}(X)$ is irreducible over $\qb$. 
If $n \equiv 5 ~({\rm mod}~ 6)$, then 
the polynomial $G_{n}(X)$ admits 
$X^2 - X +1$ as irreducible factor
in its factorization and $G_{n}(X)/(X^2 - X +1)$ is irreducible.
\end{proposition}

\begin{proof}
Selmer \cite{selmer}.
\end{proof}

\begin{proposition}
\label{closetoouane}
For all $n \geq 2$, all zeros $z_{j,n}$ and
$\theta_n$
of the polynomials
$G_{n}(X)$ have a modulus in the interval
\begin{equation}
\label{bounboun}
\Bigl[ ~1- \frac{2 \,\lo n}{ n}, ~1 + \frac{2 \,\lo 2}{n} ~\Bigr] ,
\end{equation}

(ii)~ the trinomial $G_{n}(X)$ admits a unique real root $\theta_n$ in
the interval $(0,1)$.
The sequence $(\theta_n)_{n \geq 2}$ is strictly increasing, $\lim_{n \to +\infty} \theta_n = 1$, with
$\theta_2 = \frac{2}{1+\sqrt{5}} = 0.618\ldots$,

(iii)~ 
the root $\theta_n$ 
is the unique root of smallest modulus among all the roots
of $G_{n}(X)$; 
if $n \geq 6$, 
the roots of modulus $< 1$ of $G_{n}(z)$
in the closed upper half-plane have the following properties:

(iii-1)~~ $\theta_n ~<~ |z_{1,n}|$,

(iii-2) ~ for any pair of successive indices
$j, j+1$ in $\{1, 2, \ldots, \lfloor n/6 \rfloor\}$,
$$| z_{j,n} | < |z_{j+1,n} | .$$
\end{proposition}

\begin{proof}
(i)(ii) Selmer \cite{selmer}, pp 291--292; (iii-1) Flatto, Lagarias and Poonen
\cite{flattolagariaspoonen},
(iii-2) Verger-Gaugry \cite{vergergaugry6}.
\end{proof}

\begin{proposition}
\label{rootsdistrib}
Let $n \geq 2$. 
Then (i) the number $p_n$ of roots of $G_{n}(X)$
which lie inside the open sector $\mathcal{S} =
\{ z \mid |\arg (z)| < \pi/3 \}$ is equal to
\begin{equation}
\label{pennn}
1 + 2 \lfloor \frac{n}{6} \rfloor,
\end{equation} 

(ii) the correlation between the
geometry of the roots of $G_{n}(X)$ 
which lie inside the unit disk
and the upper half-plane and their indexation is given by:
\begin{equation}
\label{rootsinside}
j \in \{1, 2, \ldots, \lfloor \frac{n}{6} \rfloor \}
~\Longleftrightarrow~ \Re(z_{j,n}) > \frac{1}{2} ~\Longleftrightarrow~ |z_{j,n}| < 1.
\end{equation} 
\end{proposition}
\begin{proof}
Proposition 3.7 in \cite{vergergaugry6}.
\end{proof}

The roots $z_{j,n}$ of $G_n$ are given 
in \cite{vergergaugry6}
by their (Poincar\'e) asymptotic expansions
\cite{vergergaugry6bis},
as a function of $n$ and $j$.
They are generically written:
${\rm Re}(z_{j,n}) = 
{\rm D}({\rm Re}(z_{j,n})) + 
{\rm tl}({\rm Re}(z_{j,n}))$,
${\rm Im}(z_{j,n}) = {\rm D}({\rm Im}(z_{j,n})) 
+ {\rm tl}({\rm Im}(z_{j,n}))$,
$\theta_n = {\rm D}(\theta_n) + 
{\rm tl}(\theta_n)$,
where "D" and "tl" stands for
 {\it ``development"}
(or {\it ``limited expansion"}, 
or {\it ``lowest order terms"})
and
"tl" for {\it ``tail"} 
(or ``remainder", or {\it ``terminant"}).

They are given at a sufficiently high order
which allows to make sufficiently
precise their positioning inside the unit disk 
and
to deduce the asymptotic expansions 
of the Mahler measures ${\rm M}(G_n)$.
The terminology {\it order} 
comes from the general theory
(Borel \cite{borel}, Copson \cite{copson}, 
Dingle \cite{dingle}, 
Erd\'elyi \cite{erdelyi}).
Here we do not need to compute the Mahler measures.

\medskip

We are only concerned with
the first root $z_{1,n}$ of $G_n$
around which a small Rouch\'e circle
(centered at $z_{1,n}$) will be set up to deduce
the existence of a zero 
of $f_{\beta}(z)$ inside it, as soon as
$n \geq 32$.

\medskip

\begin{proposition}
\label{thetanExpression}
Let $n \geq 2$.
The root $\theta_n$ can be expressed as: 
$\theta_n = {\rm D}(\theta_n) + {\rm tl}(\theta_n)$ with
${\rm D}(\theta_n) = 1 -$
\begin{equation}
\label{DthetanExpression}
\frac{\lo n}{n}
\left(
1 - \bigl(
\frac{n - \lo n}{n \, \lo n + n - \lo n}
\bigr)
\Bigl(
\lo \lo n - n 
\lo \Bigl(1 - \frac{\lo n}{n}\Bigr)
- {\rm Log} n
\Bigr)
\right)
\end{equation}
and
\begin{equation}
\label{tailthetanExpression}
{\rm tl}(\theta_n) ~=~ \frac{1}{n} \, O \left( \left(\frac{\lo \lo n}{\lo n}\right)^2 \right),
\end{equation}
with the constant $1/2$ involved in $O \left(~\right)$.
\end{proposition}

\begin{proof}
\cite{vergergaugry6} Proposition 3.1.
\end{proof}

\begin{lemma}
\label{remarkthetan}
Given the limited expansion D$(\theta_n)$ of
$\theta_n$ as in \eqref{DthetanExpression}, denote
$$\lambda_n := 1 - (1 - {\rm D}(\theta_{n}))\frac{n}{\lo n}.$$ 
Then $\lambda_n = {\rm D}(\lambda_{n}) + {\rm tl}(\lambda_{n})$, with
\begin{equation}
\label{eqqww}
{\rm D}(\lambda_n) = \frac{\lo \lo n}{\lo n} \left(\frac{1}{1+\frac{1}{\lo n}}\right), \qquad
{\rm tl}(\lambda_n) = O\left( \frac{\lo \lo n}{n}  \right)
\end{equation}
with the constant 1 in the Big O.
\end{lemma}

\begin{proof}
\cite{vergergaugry6} Lemma 3.2.
\end{proof}

The asymptotic expansions of 
those roots $z_{j,n}$
of $G_{n}(z)$ which lie 
close to the real axis
and $\{z=1\}$
are (divergent) 
sums of functions of 
a couple of 
{\em two} variables which is: 
$(n, j/\lo n)$
in the angular sector
$2 \pi \, \lo n / n > \arg z > 0$.
It is the case for $j=1$.

\begin{proposition}
\label{zedeUNmodule}
Let $n \geq 18$.
The first root $z_{1,n}$ of $G_n$ is given by:

$${\rm D}({\rm Re}(z_{1,n})) = \theta_n + \frac{2 \pi^2}{n} \left(\frac{1}{\lo n}\right)^2 
\bigl( 1+2 \lambda_n \bigr), 
$$
$${\rm D}({\rm Im}(z_{1,n})) = \frac{2 \pi \lo n}{n}  \left(\frac{1}{\lo n}\right)
\left[1 - \frac{1}{\lo n} (1 + \lambda_n)\right],$$

with
$${\rm tl}({\rm Re}(z_{1,n})) = \frac{1}{n \lo n} \left(\frac{1}{\lo n}\right)^2 O\left(
\left(\frac{\lo \lo n}{\lo n}\right)^2\right),
$$
$$
{\rm tl}({\rm Im}(z_{1,n})) = \frac{1}{n \lo n} \left(\frac{1}{\lo n}\right)
O\left(
\left(\frac{\lo \lo n}{\lo n}\right)^2\right),$$

\end{proposition}
\begin{proof}
\cite{vergergaugry6}.
\end{proof}

\subsubsection{Existence}
\label{existence}

The Salem number $\beta \in
(\theta_{n}^{-1},\theta_{n-1}^{-1})$, 
$n \geq 32$, 
is uniquely 
characterized in \eqref{fbet} by the sequence 
of exponents $(m_q)_{q \geq 0}$ in
\begin{equation}
\label{succession}
f_{\beta}(z) = -1 + z + z^n + z^{m_1} + z^{m_2}
+ z^{m_3} + \ldots = G_{n}(z) +
\sum_{q \geq 1} z^{m_q} ,
\end{equation}
with $m_0 := n$ and the 
distanciation
condition:
$m_{q+1} - m_q \geq n-1 ~~
\mbox{for all}~ q \geq 0.$
Since we do not know exactly the sequence
$(m_q)$, the problem is to make a direct
test of the Rouch\'e condition 
using the inequality
\begin{equation}
\label{rourouche}
\left|f_{\beta}(z) - G_{n}(z)\right| =
\left|\sum_{q \geq 1} z^{m_q} \right|
< 
\left| G_{n}(z) \right|
\qquad 
\mbox{for} ~z \in C_{1,n},
\end{equation}
for some circle
$C_{j,n} := \{ z : |z - z_{j,n}|= 
\frac{t_{1,n}}{n}\}$
of center $z_{1,n}$ and radius 
$t_{1,n}/n > 0$ small enough. 
To overcome this difficulty,
we proceed by considering 
the general inequality:
$$\left|f_{\beta}(z) - G_{n}(z)\right|
 =
\left|\sum_{q \geq 1} z^{m_q} \right|
\leq \sum_{q \geq 1}   | z^{m_q}|
\leq \frac{|z|^{2(n-1) + 1}}{1-|z|^{n-1}} ,
\qquad \quad |z| < 1.
$$
Then we test the Rouch\'e condition 
by the following inequality:
\begin{equation}
\label{rourouchesimple}
\frac{|z|^{2 (n - 1)+1}}{1-|z|^{n-1}} =
\frac{|z|^{2n - 1}}{1-|z|^{n-1}}
~<~ 
\left| G_{n}(z) \right|
\quad 
\mbox{for} ~z \in C_{1,n},
\end{equation}
instead of \eqref{rourouche}, 
too complicated to handle.

The problem of the choice of the radius
$t_{1,n}/n$ is
a true problem. 
The circle $C_{1,n}$ 
should not intersect
the unit circle $z=1$, nor the real axis, 
and should not contain the second root
$z_{2,n}$ of $G_n$.
On one hand, a too small
radius would lead to make impossible
the application of the Rouch\'e condition.
On the other hand,
taking larger values of $t_{1,n}/n$ 
would readily lead to
a bad or impossible localization of the zero of
$f_{\beta}(z)$. 
Indeed, we do 
not know a priori whether the unit circle is 
a natural boundary or not
for $f_{\beta}(z)$;
locating zeroes close to a natural boundary 
is a difficult problem in general.

The radius of $C_{1,n}$ 
is chosen to be 
$$\frac{t_{1,n}}{n}
=
\frac{\pi |z_{1,n}|}{n \, a_{\max}}$$
where $a_{\max}$ is determined by
the following easy lemma.

\begin{lemma}
\label{h1amaximum}
The function
$$a \to h(1,a):=
\frac{\left|1 -\exp\bigl(\frac{\pi}{a}\bigr)
\right|
\exp\bigl(\frac{-\pi}{a}\bigr)}
{\exp\bigl(\frac{\pi}{a}\bigr) 
+\left|1 -\exp\bigl(\frac{\pi}{a}\bigr)
\right|}$$
defined on $[1,+\infty)$
reaches its maximum $h(1,a_{\max}):= 0.171784\ldots$
at $a_{\max} = 5.8743\ldots$
(Figure \ref{h1a}).
\end{lemma}

\begin{figure}
\begin{center}
\includegraphics[width=6cm]{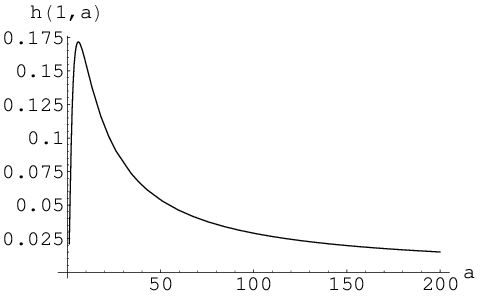}
\end{center}
\caption{
Curve of the Rouch\'e condition $a \to h(1,a)$.
}
\label{h1a}
\end{figure}

\newpage

\begin{theorem}
\label{_cercleoptiSALEM}
Let $n \geq 32$.
Denote by $C_{1,n}
:= \{z \mid |z-z_{1,n}| = 
\frac{\pi |z_{1,n}|}{n \, a_{\max}} \}$
the circle centered at the first root
$z_{1,n}$ of $G_{n}(z) = -1 + z + z^n$.
Then the 
condition of Rouch\'e 
\begin{equation}
\label{rouchecercleSALEM}
\frac{
\left|z\right|^{2 n -1}}{1 - |z|^{n-1}}
~<~
\left|-1 + z + z^n \right| , 
\qquad \mbox{for all}~ z \in C_{1,n},
\end{equation}
holds true. 
\end{theorem}

\begin{proof}
Let $a \geq 1$ and $n \geq 18$.
Denote by $\varphi := \arg (z_{1,n})$
the argument of the first root
$z_{1,n}$ (in ${\rm Im}(z) > 0$).  
Since $-1 + z_{1,n} + z_{1,n}^n = 0$, 
we have $|z_{1,n}|^n = |-1 + z_{1,n}|$.
Let us write $z= z_{1,n}+ 
\frac{\pi |z_{1,n}|}{n \, a} e^{i \psi}
=
z_{1,n}(1 + \frac{\pi}{a \, n} e^{i (\psi - \varphi)})$
the generic element belonging to 
$C_{1,n}$, with
$\psi \in [0, 2 \pi]$.
Let $X := \cos(\psi - \varphi)$.
Let us show that if the inequality
\eqref{rouchecercleSALEM} of Rouch\'e 
holds true for $X =+1$,
then it holds true
for all $X \in [-1,+1]$, that is for 
every argument $\psi \in [0, 2 \pi]$,
i.e. for every 
$z \in C_{1,n}$.
We have
$$
\left|1 + \frac{\pi}{a \,n} e^{i (\psi - \varphi)}
\right|^{n}
=
\exp\Bigl(
\frac{\pi \, X}{a}\Bigr)
\times 
\left(
1 - \frac{\pi^2}{2 a^2 \, n} (2 X^2 -1) 
+O(\frac{1}{n^2})
\right) 
$$
and
$$
\arg\left(
\Bigl(1 + \frac{\pi}{a \, n} e^{i (\psi - \varphi)}
\Bigr)^{n}\right)
=
sgn(\sin(\psi - \varphi))
\times
\left( ~\frac{\pi \, \sqrt{1-X^2}}{a}
[1 -
\frac{\pi \, X}{a \, n}
]
+O(\frac{1}{n^2})
\right)
.$$
Moreover,
$$
\left|1 + \frac{\pi}{a \, n} 
e^{i (\psi - \varphi)}
\right|
=
\left|1 + \frac{\pi}{a \, n} 
(X \pm i \sqrt{1-X^2})
\right|
=1 + \frac{\pi \, X}{a \, n} + O(\frac{1}{n^2}).
$$
with
$$\arg(1 + \frac{\pi}{a \, n} e^{i (\psi - \varphi)})
= 
sgn(\sin(\psi - \varphi)) \times
\frac{\pi \sqrt{1 - X^2}}{a \, n} 
+ O(\frac{1}{n^2}).
$$
For all $n \geq 18$, from
Proposition \ref{zedeUNmodule}, 
we have
\begin{equation}
\label{devopomainSALEM}
|z_{1,n}|
= 1 - \frac{\lo n - \lo \lo n}{n}
+ \frac{1}{n} O \left(
\frac{\lo \lo n}{\lo n}\right).
\end{equation}
from which we deduce the following equality, 
up to $O(\frac{1}{n})$ - terms,
$$
|z_{1,n}|
\,
\left|1 + \frac{\pi}{a \, n} e^{i (\psi - \varphi)}\right|
=
|z_{1,n}| .
$$
Then the left-hand side term of \eqref{rouchecercleSALEM}
is
$$\frac{
\left|z\right|^{2 n -1}}{1 - |z|^{n-1}}
=
\frac{|-1 + z_{1,n}|^2 
\left|1 + \frac{\pi}{a \, n} e^{i (\psi - \varphi)}
\right|^{2 n}}
{|z_{1,n}| \, 
\left|1 + \frac{\pi}{a \, n} e^{i (\psi - \varphi)}\right|
-
|-1 + z_{1,n}| \,
\left|1 + \frac{\pi}{a \, n} e^{i (\psi - \varphi)}\right|^{n}}$$

\begin{equation}
\label{rouchegaucheSALEM}
=
\frac{|-1 + z_{1,n}|^2 
\left(
1 - \frac{\pi^2}{a \, n} (2 X^2 -1) 
\right)
\exp\bigl(
\frac{2 \pi \, X}{a}\bigr)
}
{|z_{1,n}|
\left|1 + \frac{\pi}{a \, n} e^{i (\psi - \varphi)}\right|
-
|-1 + z_{1,n}| \,
\left(
1 - \frac{\pi^2}{2 a \, n} (2 X^2 -1) 
\right)
\exp(
\frac{\pi \, X}{a})
}
\end{equation}
up to
$\frac{1}{n} 
O \left(
\frac{\lo \lo n}{\lo n}
\right)$
-terms (in the terminant).
The right-hand side term of 
\eqref{rouchecercleSALEM} is 
$$\left|-1 + z + z^n \right|
=
\left|
-1 + z_{1,n}\Bigl(1 + \frac{\pi}{n \, a} e^{i (\psi - \varphi)}\Bigr) +
z_{1,n}^{n}
\Bigl(1 + \frac{\pi}{n \, a} e^{i (\psi - \varphi)}
\Bigr)^n
\right|
$$
$$=
\left|
-1 + z_{1,n}
(1 \pm i \frac{\pi \sqrt{1 - X^2}}{a \, n})
(1 + \frac{\pi \, X}{a \, n})
+
(1 - z_{1,n})
\left(
1 - \frac{\pi^2}{2 a^2 \, n} (2 X^2 -1) 
\right)
\right. \hspace{3cm} \mbox{}
$$
\begin{equation}
\label{rouchedroiteSALEM}
\left.
\times
\exp\bigl(
\frac{\pi \, X}{a}
\bigr) \,
\exp\Bigl(
\pm \,
i \,
\Bigl( ~\frac{\pi \, \sqrt{1-X^2}}{a}
[1 - \frac{\pi \, X}{a \, n}] 
\Bigr)
\Bigr)
+ O(\frac{1}{n^2})
\right|
\end{equation}
Let us consider
\eqref{rouchegaucheSALEM}
and
\eqref{rouchedroiteSALEM}
at the first order for the asymptotic expansions, 
i.e. up to $O(1/n)$ - terms instead of
up to 
$O(\frac{1}{n}(\lo \lo n/ \lo n))$ - terms or
$O(1/n^2)$ - terms.
\eqref{rouchegaucheSALEM} becomes:
$$\frac{|-1+z_{1,n}|^2 \exp(\frac{2 \pi X}{a})}
{|z_{1,n}| - |-1+z_{1,n}| \exp(\frac{\pi X}{a})}$$
and \eqref{rouchedroiteSALEM} is equal to:
$$|-1 + z_{1,n}|
\left|
1 -
\exp\bigl(
\frac{\pi \, X}{a}
\bigr) \,
\exp\Bigl(
\pm \,
i \,
\frac{\pi \, \sqrt{1-X^2}}{a} 
\Bigr)
\right|
$$
and is independent of the sign of 
$\sin(\psi - \varphi)$.
Then
the inequality \eqref{rouchecercleSALEM} is 
equivalent to
\begin{equation}
\label{roucheequiv1SALEM}
\frac{|-1+z_{1,n}|^2 \exp(\frac{2 \pi X}{a})}
{|z_{1,n}| - |-1+z_{1,n}| \exp(\frac{\pi X}{a})}
<
|-1+z_{1,n}| \, 
\left|
1 -
\exp\bigl(
\frac{\pi \, X}{a}
\bigr) \,
\exp\Bigl(
\pm \,
i \,
\frac{\pi \, \sqrt{1-X^2}}{a} 
\Bigr)
\right|
,
\end{equation}
and \eqref{roucheequiv1SALEM} to
\begin{equation}
\label{amaximalfunctionXSALEM}
\frac{|-1 + z_{1,n}|}{|z_{1,n}|}
~  < ~ \, 
\frac{\left|
1 -
\exp\bigl(
\frac{\pi \, X}{a}
\bigr) \,
\exp\Bigl(
 \,
i \,
\frac{\pi \, \sqrt{1-X^2}}{a} 
\Bigr)
\right|
\exp\bigl(
\frac{-\pi \, X}{a}
\bigr)}{\exp\bigl(
\frac{\pi \, X}{a}
\bigr) +\left|
1 -
\exp\bigl(
\frac{\pi \, X}{a}
\bigr) \,
\exp\Bigl(
 \,
i \,
\frac{\pi \, \sqrt{1-X^2}}{a} 
\Bigr)
\right|} = \kappa(X,a).
\end{equation}
The right-hand side function
$\kappa(X,a)$ is 
a function of $(X, a)$, 
on $[-1, +1] \times [1, +\infty)$.
which is strictly decreasing for any
fixed $a$, 
and reaches its minimum
at $X=1$; this minimum is always 
strictly positive. 
Consequently 
the inequality of Rouch\'e
\eqref{rouchecercleSALEM} will be satisfied
on $C_{1,n}$ once it is 
satisfied at $X = 1$, as claimed.

Hence, up to
$O(1/n)$-terms, the Rouch\'e condition
\eqref{amaximalfunctionXSALEM}, 
for any fixed $a$,
will be satisfied (i.e. for any
$X \in [-1,+1]$)
by the set of integers
$n = n(a)$ for which $z_{1,n}$
satisfies:
\begin{equation}
\label{amaximalfunctionSALEM}
\frac{|-1 + z_{1,n}|}{|z_{1,n}|} 
< \kappa(1,a) 
=
\frac{\left|
1 -
\exp\bigl(
\frac{\pi}{a}
\bigr)
\right|
\exp\bigl(
\frac{-\pi}{a}
\bigr)}{\exp\bigl(
\frac{\pi}{a}
\bigr) +\left|
1 -
\exp\bigl(
\frac{\pi}{a}
\bigr)
\right|} ,
\end{equation}
equivalently, from Proposition 
\ref{zedeUNmodule},
\begin{equation}
\label{amaximalfunctionnnSALEM}
\frac{\lo n - \lo \lo n}{n} 
< \frac{\kappa(1,a)}
{1 + \kappa(1,a)} .
\end{equation}
In order to 
obtain 
the largest possible range 
of values of $n$,
the value of $a \geq 1$ has to be chosen
such that $a \to \kappa(1,a)$ is maximal
in \eqref{amaximalfunctionnnSALEM}.
From Lemma \ref{h1amaximum} 
we take $a = a_{\max}$.

The slow decrease
of the functions of the variable $n$
involved in 
the terminants when $n$ tends to infinity,
as a factor of uncertainty on
\eqref{amaximalfunctionnnSALEM},
has to be taken into account
in \eqref{amaximalfunctionnnSALEM}.
It amounts to check 
numerically 
whether \eqref{rouchecercleSALEM} 
is satisfied 
for the small values
$18 \leq n \leq 100$
for $a = a_{\max}$, or not. 
Indeed, for the
large enough values of $n$,
the inequality
\eqref{amaximalfunctionnnSALEM}
is satisfied since
$\lim_{n \to \infty}
\frac{\lo n - \lo \lo n}{n} = 0$.
On the computer, 
the critical threshold
of $n = 32$  
is easily calculated, with
$(\lo 32 - \lo \lo 32)/32 = 
0.0694628\ldots$.
Then
$$\frac{\lo n - \lo \lo n}{n}
<
\frac{\kappa(1,a_{\max})}
{1 + \kappa(1,a_{\max})} = 0.146447\ldots\qquad
\mbox{for all
$n \geq 32$} .$$
Let us note that
the last inequality
also holds for some values of
$n$ less
than $ 32$.
\end{proof}

\begin{corollary}
\label{smallestSALEM}
Let $\beta > 1$ be a Salem number
such that
$\dyg(\beta) \geq 32$ (if any). Then
the Parry Upper function
$f_{\beta}(z)$
admits a simple nonreal zero 
of modulus
$< 1$ in the 
open disk $\{z : |z-z_{1,n}|< 
\frac{\pi |z_{1,n}|}{n \, a_{\max}}\}$.
\end{corollary}

\begin{proof}
The polynomial $G_{n}(z)$ has
simple roots.
Since
\eqref{rouchecercleSALEM} is satisfied, 
the Theorem of Rouch\'e states that
$f_{\beta}(z)$
and
$G_{n}(z) = -1 + z + z^n$
have the same number of roots, 
counted with
multiplicities, in the open disk
$\{z : |z-z_{1,n}|< 
\frac{\pi |z_{1,n}|}{n \, a_{\max}}\}$.
We deduce the existence of an unique simple zero
of $f_{\beta}(z)$ inside this disk.
\end{proof}

Let us denote by $\omega_{1,n}$ this zero of 
$f_{\beta}(z)$. To summarize, for
$n \geq 32$, 
$\theta_{n}^{-1} < \beta < \theta_{n-1}^{-1}$,
taking the 
complex-conjugates, we have the 2 tri-uplets 
(represented schematically in Fig. \ref{example649}):
$$\{\overline{z_{1,n}}, \theta_{n}, z_{1,n}\},~~ 
\{\overline{\omega_{1,n}}, \beta^{-1}, \omega_{1,n}\}$$
which are very close in a narrow symmetrical
angular sector
containing $\{z=1\}$ and satisfy:
$$f_{\beta}(1/\beta) =
 f_{\beta}(\omega_{1,n}) =
 f_{\beta}(\overline{\omega_{1,n}}) = 0$$
 and
$$ f_{\theta_{n}^{-1}}(\theta_n) =
f_{\theta_{n}^{-1}}(z_{1,n})=
f_{\theta_{n}^{-1}}(\overline{z_{1,n}}) =0.
 $$
In the sequel we prove that $\omega_{1,n}$
is a conjugate of $\beta^{-1}$, then of $\beta$.

\begin{figure}[h]
\begin{center}
a)
\includegraphics[width=3cm]{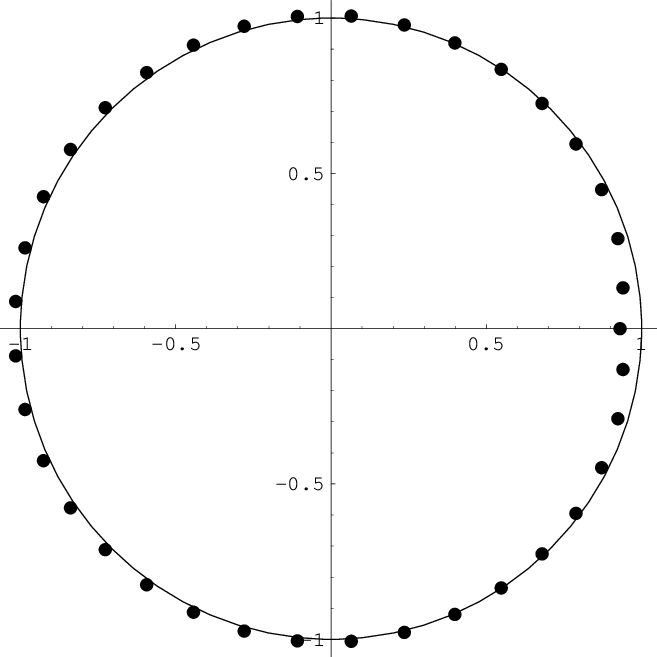}
b)
\includegraphics[width=3cm]{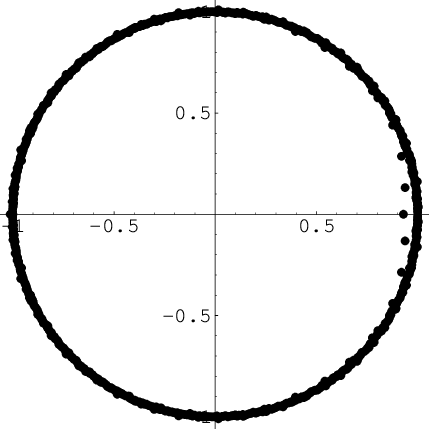}
\end{center}
\caption{{\small
In a) and b) the zeroes
of $G_n$ and $f_{\beta}$, resp., are represented by 
black bullets, with lenticularity appearing
(symmetrically with respect to $\rb$)
in the angular sector $|\arg z| < \pi/18$, 
about the unit circle in $\cb$. In a)
the first zero $z_{1,n}$ of $G_n$ (here $n=37$) is
the first one in this sector having positive
imaginary part. 
In b) the zero
$\omega_{1,n}$ 
is obtained by a 
very slight deformation of 
$z_{1,n}$
according to Theorem \ref{_cercleoptiSALEM}.
The other roots 
of $f_{\beta}$ can be found in a narrow annular 
neighbourhood of 
$|z|=1$.}}
\label{example649}
\end{figure}

\subsection{Rewriting trails}
\label{S4.2}

Let us assume that
$f_{\beta}(\omega_{1,n}) =0$
with the property that the minimal polynomial
$P_{\beta}$ of $\beta$
satisfies $P_{\beta}(\omega_{1,n})\neq 0$,
and show the contradiction.
Denote $\nu := |P_{\beta}(\omega_{1,n})| > 0$.
Let us consider
the $s$-th polynomial section
$S_{s}(z)=-1 + \sum_{j=1}^{s} t_j z^j$ 
of $f_{\beta}(z)$, where the integer
$s$, taken large enough, will be
fixed below. The vector coefficient
$(t_j)$ is as in \eqref{succession}.

The $s$-th polynomial section
$S_{s}(z)$ admits a unique real zero
in $(0,1)$.
Indeed $S_{s}(0)=-1$,
$S_{s}(1) > 1$, and
the derivative  of the restriction
of $S_{s}(z)$ to $[0,1]$ is positive
on $[0,1]$.
The polynomial
$S^{*}_{s}(z)$, reciprocal polynomial
of $S_{s}(z)$, admits a unique
real zero, say 
$\gamma_s, > 1$. 
We have: $\deg(S_s) \leq s$
and $\lim_{s \to \infty} 
\gamma_{s}^{-1} = \beta^{-1}$.
The real number
$\gamma_{s}$
is a nonreciprocal
algebraic integer which is such that 
$1 < \gamma_s < \beta$:
indeed, $y \to S_{s}(y)$ is strictly increasing
on $(0,1)$ and 
$S_{s}(\beta^{-1}) = 
-1 + \sum_{j=1}^{s} t_j \beta^{-j}
=
f_{\beta}(\beta^{-1}) 
-\sum_{j=s+1}^{\infty} t_j \beta^{-j}
=-\sum_{j=s+1}^{\infty} t_j \beta^{-j} <0
$,
so that $\beta^{-1} < \gamma_{s}^{-1}$.
There exists an integer, say $W_{\nu}$,
such that: 
$s \geq W_{\nu} \Longrightarrow
|P_{\beta}(\gamma_{s}^{-1})| < 
\min\{1, \nu/2\}.$
In the following we take
$s \geq W_{\nu}$.

Before going further, let us recall the 
general factorization 
of a Parry Upper function $f_{\alpha}(z)$
for 
$\alpha$ any simple Parry number.

\begin{theorem}
\label{nonreciprocalpart}
Let $\alpha$ be a simple Parry number,
with $\dyg(\alpha) \geq 3$.
Let
$f_{\alpha}$ 
denote its Parry Upper function with factorization
in $\zb[x]$:
$$f_{\alpha}(x) = A(x) B(x) C(x) =
-1 + x + x^n +
x^{m_1} + x^{m_2} + \ldots + x^{m_s},$$
where $s \geq 1$, $m_1 - n \geq n-1$, 
$m_{j+1}-m_j \geq n-1$ for $1 \leq j < s$,
where 
\begin{itemize}
\item $A$ is 
the cyclotomic component,  
\item $B$ the reciprocal noncyclotomic component,
\item $C$ the nonreciprocal part.
\end{itemize}
Then $C$ is irreducible. Moreover
$C$ has no root of modulus 1.
\end{theorem}

\begin{proof}
\cite{dutykhvergergaugry}
\cite{dutykhvergergaugry3}.
\end{proof}
 
Then, inside the circle
$C_{1,n}$ the polynomial section
$S_{s}(z)$ has a unique zero
(same proof as for 
Theorem \ref{_cercleoptiSALEM});
let us denote it by
$r_s$. We have:
$\lim_{s \to \infty} r_s = 
\omega_{1,n}$
and $r_s$ is equal to
the conjugate
$\sigma_{s}(\gamma_{s}^{-1})$
of $\gamma_{s}^{-1}$
for some $\sigma_s$ which is the conjugation
relative to the irreducible nonreciprocal
(never trivial) part $C$ of $S_s$ by 
Proposition \ref{nonreciprocalpart}.
We have: $C(\gamma_{s}^{-1}) =
C(r_s) =0$.
Denote by $s_c = \deg(C)$ the degree
of the component $C$.
The irreducible 
polynomials $P_{\beta}(X)$ and
$C(X)$ are coprime: indeed 
the first one is reciprocal
while the second one is nonreciprocal.
The integer $s_c$ is a function of $s$.
Denote
$$d := \deg (P_{\beta})
\qquad {\rm and} \qquad
H:= \max_{j =1, \ldots, d-1} \{|a_j|\} \geq 1$$
the (na\"ive) height of $P_{\beta}(X)
= 1 + \sum_{j=1}^{d-1} a_j X^j + X^d$.

\bigskip 
 
Let us now define the notion of {\em rewriting trail}
and justify its use. On one side, the Salem
number $\beta$ is entirely determined by
its minimal polynomial $P_{\beta}(X)$, from
Commutative Algebra. On the other side,
$\beta$ is completely determined by the 
Parry Upper function
$f_{\beta}(z)$ coming from the dynamical system. 
A priori, the two analytic functions
$z \to P_{\beta}(z)$ and $z \to f_{\beta}(z)$
have very different coefficient vectors, the second
one never being a polynomial, by
Proposition \ref{fbetainfinie}. But the 
``object" $\beta$ is the same in both cases.

These two functions give rise to
two distinct $\beta$-representations, 
from the equations:

$$P_{\beta}(1/\beta) = P_{\beta}(\beta)=0 \qquad 
\qquad \mbox{by the minimal polynomial}$$
and
$$f_{\beta}(1/\beta)=0 \qquad  
 \mbox{by the numeration dynamical system and}~ \zeta_{\beta}(z).$$
To prove that $\omega_{1,n}$ is a conjugate
of $\beta^{-1}$, to pass from the first 
$\beta$-representation to the second 
$\beta$-representation is needed. We cannot proceed
directly, and a sequence of intermediate
$\beta$-representations is required.

In this subsection, below, we
substitute $\beta$ by $\gamma_s$ and 
$f_{\beta}(z)$ by $S_{s}(z)$, and construct 
step by step the intermediate
$\gamma_s$-representations.
These intermediate
$\gamma_s$-representations constitute a 
rewriting trail. There is no unicity.
Then, in the next subsection,
the limit when $s$ tends to infinity 
is taken to conclude.
 
 \medskip
 
Let us construct the
rewriting trail  
from ``$S_{s}$" to
``$P_{\beta}$", at $\gamma_{s}^{-1}$
\cite{dutykhvergergaugry2}.
The starting point is the identity
$1 = 1$, to which we add 
$0=  S_{\gamma_{s}}(\gamma_{s}^{-1})$
in the (rhs) right hand side.
Then we define the rewriting trail from
the R\'enyi 
$\gamma_{s}^{-1}$-expansion of 1
\begin{equation}
\label{gammasexpansionSs}
1=1+S_{\gamma_{s}}(\gamma_{s}^{-1})
=
t_{1}\gamma_{s}^{-1}  +t_2 \gamma_{s}^{-2} + \ldots 
+ t_{s-1} \gamma_{s}^{-(s-1)} +t_s \gamma_{s}^{-s}
\end{equation}
(with $t_1 = 1,
t_2 = t_3 = \ldots =
t_{n-1} = 0,
t_n = 1,$ etc) to
\begin{equation}
\label{equa1Pgammas}
- a_1 \gamma_{s}^{-1} - a_2 \gamma_{s}^{-2} + \ldots - 
a_{d-1} \gamma_{s}^{-(d-1)} - \gamma_{s}^{-d}
= 1 - P_{\beta}(\gamma_{s}^{-1}),
\end{equation}
by ``restoring" the digits
of $1 - P_{\beta}(X)$ one after the other,
from the left.

We obtain a sequence
$(A'_{q}(X))_{q \geq 1}$ of
rewriting polynomials involved
in this rewriting trail; 
for $q \geq 1$,
$A'_q \in \zb[X]$,
$\deg(A'_q) \leq q$
and $A'_{q}(0) = 1$.

At the first step we add $0=
-(-a_1 - t_1) \gamma_{s}^{-1} S_{\gamma_{s}}^{*}(\gamma_{s}^{-1})$;
and we obtain 
$$1= -a_1 \gamma_{s}^{-1}$$
$$+(-(-a_1 -t_1) t_1 + t_2) \gamma_{s}^{-2}
+(-(-a_1 -t_1) t_2 + t_3) \gamma_{s}^{-3} + \ldots
$$
so that the height of the polynomial
$$(-(-a_1 -t_1) t_1 + t_2) X^{2}
+(-(-a_1 -t_1) t_2 + t_3) X^{3} + \ldots
$$
is $\leq H+2$.
At the second step we add
$0=
-(-a_2 - (-(-a_1 -t_1) t_1 + t_2)) \gamma_{s}^{-2} 
S_{\gamma_{s}}^{*}(\gamma_{s}^{-1})
$.
Then we obtain
$$1= -a_1 \gamma_{s}^{-1} - a_2 \gamma_{s}^{-2}$$
$$-
[(-a_2 - (-(-a_1 -t_1) t_1 + t_2))t_1
+ (-(-a_1 -t_1) t_2 + t_3)] \gamma_{s}^{-3}+\ldots
$$
where the height of the polynomial
$$-[(-a_2 - (-(-a_1 -t_1) t_1 + t_2))t_1
+ (-(-a_1 -t_1) t_2 + t_3)] X^{3}+\ldots
$$
is $\leq H + (H+2)+(H+2)=3 H+4$.
Iterating this process $d$ times 
we obtain
$$1= -a_1 \gamma_{s}^{-1} - 
a_2 \gamma_{s}^{-2} -\ldots
- a_d \gamma_{s}^{-d}$$
$$+~~
polynomial ~~remainder~~ in~~ \gamma_{s}^{-1}.
$$
Denote by $V(\gamma_{s}^{-1})$
this polynomial remainder in $\gamma_{s}^{-1}$,
for some $V(X) \in \zb[X]$,
and $X$ specializing in $\gamma_{s}^{-1}$.
If we denote the upper bound of the
height of the polynomial remainder
$V(X)$, 
at step $q$, by $\lambda_q H + v_q$, 
we readily
deduce: $v_q = 2^q$, and
$\lambda_{q+1} = 2 \lambda_{q} +1$, $q \geq 1$,
with $\lambda_1 = 1$; then 
$\lambda_{q} = 2^{q}-1$.

To summarize,
the first 
rewriting polynomials of the
sequence
$(A'_{q}(X))_{q \geq 1}$ 
involved in this rewriting trail
are
$$A'_{1}(X) = 
-1 - (-a_1 - t_1) X,$$ 
$$A'_{2}(X) 
= 
-1 - (-a_1 - t_1) X -
(-a_2 - (-(-a_1 -t_1) t_1 + t_2)) X^2 , \quad {\rm etc}.$$

\

For $q \geq \deg(P_{\beta})$, all the coefficients 
of $P_{\beta}$ are ``restored"; denote
by
$(h_{q,j})_{j=0,1,\ldots,s-1}$ the $s$-tuple
of integers produced by this rewriting trail,
at step $q$. It is such that
\begin{equation}
\label{AprimeSPreste}
A'_{q}(\gamma_{s}^{-1}) S_{\gamma_{s}}^{*}(\gamma_{s}^{-1})
=
-P(\gamma_{s}^{-1}) + \gamma_{s}^{-q-1}
\Bigl(
\sum_{j=0}^{s-1} h_{q,j} \gamma_{s}^{-j}
\Bigr).
\end{equation}
Then take $q=d$.
The (lhs) left-and side 
of \eqref{AprimeSPreste} is equal to 
$0$.
Thus 
$$P(\gamma_{s}^{-1}) =
 \gamma_{s}^{-d-1}
\Bigl(
\sum_{j=0}^{s-1} h_{d,j} \gamma_{s}^{-j}
\Bigr)
\qquad
\Longrightarrow
\qquad
P(\gamma_{s}) =
\sum_{j=0}^{s-1} h_{d,j} \gamma_{s}^{-j-1}.
$$
The height
of the polynomial
\begin{equation}
\label{Wpolynomial}
W(X) :=\sum_{j=0}^{s-1} h_{d,j} X^{j+1}
\qquad {\rm is}\qquad \leq
 (2^d -1) H + 2^d,
 \end{equation} 
 and is independent of $s
 \geq W_v$.

\

For any $s \geq W_{\nu},$ 
let us observe that
$- P_{\beta}(\gamma_{s}^{-1})
$ is $> 0$, and that the sequence
$(\gamma_{s}^{-1})_s$ is decreasing.
Indeed, 
the polynomial function
$x \to P_{\beta}(x)$ is positive
on $(0, \beta^{-1})$, vanishes
at $\beta^{-1}$, 
and changes its sign
for
$x > \beta^{-1}$, 
so that 
$P_{\beta}(\gamma_{s}^{-1}) < 0$.
We have: $\lim_{s \to \infty}
P_{\beta}(\gamma_{s}^{-1}) =
P_{\beta}(\beta^{-1})=0$.

Let us use the 
$\gamma_{s}$-transformation and the
greedy (R\'enyi) $\gamma_{s}$-expansion of
$- P_{\beta}(\gamma_{s}^{-1})$: 
there exists
an unique sequence of integers
$(\widehat{t_i})_{i \geq 1} \not\equiv (0)$ 
in the alphabet
$\{0, 1\}$ 
such that
\begin{equation}
\label{boutgammashift}
- P_{\beta}(\gamma_{s}^{-1})
= \frac{\widehat{t_1}}{\gamma_s} +
\frac{\widehat{t_2}}{\gamma_{s}^{2}} +
\frac{\widehat{t_3}}{\gamma_{s}^{3}} +
\ldots .
\end{equation}
The integers $\widehat{t_i}$
are given by the
$\gamma_{s}$-transformation
$T_{\gamma_{s}}: [0,1] \to [0,1],
x \to \{\gamma_{s} x\}$. 
Explicitely, 
the digits 
are

$\widehat{t_1} = \lfloor \gamma_s (- P_{\beta}(\gamma_{s}^{-1})) \rfloor$,

$\widehat{t_2} 
= \lfloor \gamma_s \{\gamma_s (- P_{\beta}(\gamma_{s}^{-1})) \} \rfloor$,

$\widehat{t_3} 
= \lfloor \gamma_s \{\gamma_s \{\gamma_s (- P_{\beta}(\gamma_{s}^{-1})) \} \} \rfloor, \, 
\ldots$ \,, and depend upon $\gamma_{s}$.

\noindent
Since $\lim_{s \to \infty}
P_{\beta}(\gamma_{s}^{-1})
=
P_{\beta}(\beta^{-1}) = 0$
there exists an increasing  sequence 
$(u_s)_{s \geq W_{\nu}}$ of positive integers,
satisfying
$\widehat{t_1}= \widehat{t_2}
=\ldots = \widehat{t_{u_s -1}} = 0$,
$\widehat{t_{u_s}}=1$,
such that 
the
identity 
between $- P_{\beta}(\gamma_{s}^{-1})$ and
its greedy expansion holds, as:
\begin{equation}
\label{boutgammashift_uZERO}
- P_{\beta}(\gamma_{s}^{-1})
= \frac{\widehat{t_{u_s}}}{\gamma_{s}^{u_s}} +
\frac{\widehat{t_{u_s + 1}}}{\gamma_{s}^{u_s + 1}} +
\frac{\widehat{t_{u_s + 2}}}{\gamma_{s}^{u_s + 2}} +
\ldots .
\end{equation}
The sequence $(u_s)$ is defined by the bounds
\begin{equation}
\label{bound_us}
\bigl|
\beta^{u_s} (P_{\beta}(\gamma_{s}^{-1}))
\bigr|
\geq
\bigl|
\gamma_{s}^{u_s} (P_{\beta}(\gamma_{s}^{-1}))
\bigr| \geq 1
\end{equation}
and
$$\bigl|
\gamma_{s}^{u_s} (P_{\beta}(\gamma_{s}^{-1}))
\bigr| \leq 
\frac{1}{1-\gamma_{s}^{-1}}.$$
Therefore, in \eqref{boutgammashift_uZERO},

\begin{equation}
\label{hgammasSeries}
\widehat{t_i} \in \{0,1\}, \quad i \geq 1,
\qquad
{\rm and}
\qquad
\lim_{s \to +\infty} u_s = +\infty .
\end{equation}

Now the lhs of \eqref{boutgammashift_uZERO} belongs to
$\mathbb{Q}(\gamma_{s})$. 
For conjugating \eqref{boutgammashift_uZERO}
by $\sigma_s$, if 
the image by
$\sigma_s$ of the lhs of
\eqref{boutgammashift_uZERO}
belongs to $\mathbb{Q}(r_s)$,
there are three cases for
the conjugation of the rhs of \eqref{boutgammashift_uZERO}:
\begin{enumerate}
\item[(i)] the rhs of \eqref{boutgammashift_uZERO} is finite (ends in infinitely many zeroes),
\item[(ii-1)] 
the rhs of \eqref{boutgammashift_uZERO} is eventually periodic (infinite and
ultimately periodic),
\item[(ii-2)] the rhs of \eqref{boutgammashift_uZERO} is infinite
and not eventually periodic.
\end{enumerate}
\

{\bf Case (i)}:
say  that
$\frac{\widehat{t_{u_s}}}{\gamma_{s}^{u_s}} +
\frac{\widehat{t_{u_s + 1}}}{\gamma_{s}^{u_s + 1}} +
\ldots+
\frac{\widehat{t_{u_s + N}}}{\gamma_{s}^{u_s + N}}$
is the rhs of \eqref{boutgammashift_uZERO}.
Then its image by $\sigma_s$ is
$\widehat{t_{u_s}} \,r_{s}^{u_s} +
\widehat{t_{u_s + 1}} \,r_{s}^{u_s + 1} +
\ldots+
\widehat{t_{u_s + N}} \, r_{s}^{u_s + N}$
and we have the equality
$$P_{\beta}(r_{s})
=
\sigma_{s}\left(- \sum_{j=0}^{s_c-1}
h''_{j} \,\gamma_{s}^{-j-d-2}
\right)=
- \sum_{j=0}^{s_c-1}
h''_{j} \,r_{s}^{j+d+2}
$$
$$=
\sigma_{s} \left( 
\frac{\widehat{t_{u_s}}}{\gamma_{s}^{u_s}} +
\frac{\widehat{t_{u_s + 1}}}{\gamma_{s}^{u_s + 1}} +
\ldots +
\frac{\widehat{t_{u_s + N}}}{\gamma_{s}^{u_s + N}} 
 \right)
=
\widehat{t_{u_s}} r_{s}^{u_s} +
\widehat{t_{u_s + 1}} r_{s}^{u_s + 1} +
\ldots +
\widehat{t_{u_s + N}} r_{s}^{u_s + N} .$$ 
Conjugation by $\sigma_s$ is done 
term by term.

\

{\bf Case (ii-1)}: 
the rhs of \eqref{boutgammashift_uZERO} 
is eventually periodic. Let us write it
$$
=
\frac{\widehat{t_{u_s}}}{\gamma_{s}^{u_s}} +
\frac{\widehat{t_{u_s + 1}}}{\gamma_{s}^{u_s + 1}} +
\ldots+
\frac{\widehat{t_{u_s + N}}}{\gamma_{s}^{u_s + N}}
+
\sum_{i=0}^{\infty}
\left(
\frac{\widehat{t_{u_s + N + 1}}}{\gamma_{s}^{u_s + N + i q + 1}} +
\frac{\widehat{t_{u_s + N + 2}}}{\gamma_{s}^{u_s + N + i q  +  2}} +
\ldots+
\frac{\widehat{t_{u_s + N + q}}}{\gamma_{s}^{u_s + N+ i q +q}}
\right).$$
The period is not equal to zero. The period length is
$q$.
We have: $|r_s|=|\sigma_{s}(\gamma_{s}^{-1})| < 1$. 
Then it is equal to
$$=
\frac{\widehat{t_{u_s}}}{\gamma_{s}^{u_s}} +
\ldots+
\frac{\widehat{t_{u_s + N}}}{\gamma_{s}^{u_s + N}}
+
\sum_{i=0}^{\infty}
\gamma_{s}^{- i q }
\left(
\frac{\widehat{t_{u_s + N + 1}}}{\gamma_{s}^{u_s + N  + 1}} +
\ldots+
\frac{\widehat{t_{u_s + N + q}}}{\gamma_{s}^{u_s + N +q}}
\right)$$
$$=
\frac{\widehat{t_{u_s}}}{\gamma_{s}^{u_s}} +
\ldots+
\frac{\widehat{t_{u_s + N}}}{\gamma_{s}^{u_s + N}}
+
\frac{1}{1-
\gamma_{s}^{- q }}
\left(
\frac{\widehat{t_{u_s + N + 1}}}{\gamma_{s}^{u_s + N  + 1}} +
\ldots+
\frac{\widehat{t_{u_s + N + q}}}{\gamma_{s}^{u_s + N +q}}
\right)$$
and its image by $\sigma_s$ is
$$=
\widehat{t_{u_s}} r_{s}^{u_s} +
\ldots+
\widehat{t_{u_s + N}} r_{s}^{u_s + N}
+
\frac{1}{1-
r_{s}^{q }}
\left(
\widehat{t_{u_s + N + 1}} r_{s}^{u_s + N  + 1} 
+
\ldots+
\widehat{t_{u_s + N + q}} r_{s}^{u_s + N +q}
\right)$$
$$=
\widehat{t_{u_s}} r_{s}^{u_s} +
\ldots+
\widehat{t_{u_s + N}} r_{s}^{u_s + N}
+
\sum_{i=0}^{\infty}
\left(
\widehat{t_{u_s + N + 1}} r_{s}^{u_s + N+ i q + 1} 
+
\ldots+
\widehat{t_{u_s + N + q}} r_{s}^{u_s + N + i q +q}
\right) .$$
The series can be conjugated term by term
by $\sigma_s$;
in this case we have the identity
$$P_{\beta}(r_{s})
=
\sigma_{s} \left( 
\frac{\widehat{t_{u_s}}}{\gamma_{s}^{u_s}} +
\frac{\widehat{t_{u_s + 1}}}{\gamma_{s}^{u_s + 1}} +
\ldots +
 \right)
=
\widehat{t_{u_s}} r_{s}^{u_s} +
\widehat{t_{u_s + 1}} r_{s}^{u_s + 1} +
\ldots  
.$$

\

{\bf Case (ii-2)}: 
if the rhs
of \eqref{boutgammashift_uZERO}
is 
not eventually periodic its conjugation
by $\sigma_s$ cannot be done term by term.
This difficulty
is overcome by enlarging 
the alphabet $\mathcal{A} =\{0,1\}$ to a bigger alphabet
$\mathcal{B}$ and by replacing
the R\'enyi expansion
by a
$(\gamma_{s}, \mathcal{B})$-eventually periodic 
representation of $- P_{\beta}(\gamma_{s}^{-1})$.
This method of enlargement of the alphabet
\cite{dutykhvergergaugry2}
is made possible as
a consequence of the Theorem of Kala-Vavra
recalled in the next subsection. It has to be noted
that Theorem \ref{nonreciprocalpart} implies that
we are in the domain of applicability
of Kala-Vavra's Theorem, because $\gamma_{s}$ 
has no conjugates
on the unit circle. It is the key point.

\subsection{Kala-Vavra Theorem and Galois identification}
\label{S4.3}

Let us recall
the definitions.
The $(\delta, \mathcal{B})$-representations for
a given $\delta \in \mathbb{C}$,
$|\delta| > 1$ and
a given alphabet
$\mathcal{B} \subset \mathbb{C}$ finite,
are expressions of the form
$\sum_{k \geq -L} a_k \delta^{-k}$,
$a_k \in \mathcal{B}$, 
for some integer $L$. We denote
$$
{\rm Per}_{\mathcal{B}}(\delta)
:=
\{x \in \mathbb{C}
:
x \,\,{\rm has \,an \,eventually \,periodic}\,
(\delta, \mathcal{B}){\rm -representation}\}.
$$

\begin{theorem}[Kala - Vavra \cite{kalavavra}, 
\cite{bakermasakovapelantovavavra}]
\label{kalavavra}
Let $\delta \in \mathbb{C}$ be an 
algebraic number of degree $d$, 
$|\delta|>1$, 
and
$a_d x^d - a_{d-1} x^{d-1}
- \ldots - a_1 x -a_0 \in \mathbb{Z}[x]$,
$a_0 a_d \neq 0$, be its minimal
polynomial. Suppose that
$|\delta'| \neq 1$ for any conjugate 
$\delta'$ of $\delta$,
Then there exists a finite alphabet $\mathcal{B} \subset \mathbb{Z}$ such that
$$\mathbb{Q}(\delta) = {\rm Per}_{\mathcal{B}}(\delta).$$
\end{theorem}

Let us apply Theorem \ref{kalavavra}
to $\delta = \gamma_{s}$. By Proposition 5
in \cite{dutykhvergergaugry} $\gamma_s$
has no conjugate of modulus 1. 
Therefore there exists a finite alphabet
$\mathcal{B} \subset \mathbb{Z}$ 
such that the lhs of \eqref{boutgammashift_uZERO}
be identified with a
$(\gamma_{s}, \mathcal{B})$- representation
which is eventually periodic, for some
integer $u_s \in \zb$:
\begin{equation}
\label{hgammasSeries_alphabetA}
- P_{\beta}(\gamma_{s}^{-1})
= 
\frac{\widehat{t_{u_s }}}{\gamma_{s}^{u_s }} +
\frac{\widehat{t_{u_s + 1 }}}{\gamma_{s}^{u_s + 1 }} +
\frac{\widehat{t_{u_s + 2 }}}{\gamma_{s}^{u_s + 2 }} +
\ldots ~~.
\end{equation}
Being eventually periodic,
the representation
\eqref{hgammasSeries_alphabetA}
can now be
conjugated term by term by $\sigma_s$, 
since $|\sigma_{s}(\gamma_{s}^{-1})|<1$,
as in case (ii-1).
In \eqref{hgammasSeries_alphabetA} the digits
$\widehat{t_{i }}$ belong to 
a symmetrical alphabet
$\mathcal{B} =\{-m, \ldots, 0, \ldots, m\}$;
the integer $m$ is provided
by the rewriting trail, given by
the rewriting trail 
and
\eqref{Wpolynomial}: we have
$m = \lceil 2((2^d -1) H + 2^d)/3 \rceil$.

\

Indeed, by 
Theorem \ref{kalavavra}
there exist 
a preperiod $R(X) \in \mathcal{B}[X]$, 
a period
$T(X) \in \mathcal{B}[X]$ such that
$$\widehat{W}(\gamma_{s}^{-1})
:=
-P_{\beta}(\gamma_{s}) 
= R(\gamma_{s}^{-1})
+ \gamma_{s}^{-\deg R -1} 
\sum_{j=0}^{\infty} 
\frac{1}{\gamma_{s}^{j (\deg T + 1)}}
T(\gamma_{s}^{-1}),
$$
both polynomials $R$ and $T$ depending upon $s$.
Since the relation
$$S_{\gamma_{s}}(\gamma_{s}^{-1})
=
-1 +t_{1} \gamma_{s}^{-1} 
+t_2 \gamma_{s}^{-2} + \ldots 
+ t_{s-1} \gamma_{s}^{-s+1} 
+t_s \gamma^{-s} = 0$$ 
holds, we may assume
$\deg R \leq s-1$,
$\deg T \leq s-1$.
Then, for $X$ specialized at $\gamma_{s}^{-1}$,
we have the identity
\begin{equation}
\label{Wrepresentation}
\widehat{W}(X) = R(X) +
X^{L}\frac{T(X)}{1-X^r}
\end{equation}
for some positive integers $L, r$.
The height of $(1-X^r) \widehat{W}(X)$ is 
$\leq 2 ((2^d -1) H + 2^d)$
and, with 
$\mathcal{B}$
assumed $=\{-m, \ldots,0,\ldots,+m\}$,
the height of 
$(1-X^r) R(X) + X^L T(X)$ is less than
$3 m$. Therefore $m$ is 
$\leq 2 ((2^d -1) H + 2^d)/3$.
We can take 
$$m = \lceil 2((2^d -1) H + 2^d)/3 \rceil.$$
The alphabet 
$\mathcal{B} = \{-m, \ldots, m\}$
only
depends upon the degree $d$ and the height
$H$ of the polynomial $P_{\beta}$, and
does not depend upon $s$.

\

We now assume
$0 \neq |P_{\beta}(\gamma_s)| \ll 1$.
The
$(\gamma_{s}, \mathcal{B})$-eventually periodic
representation
of $-P_{\beta}(\gamma_{s})$
starts as
$$-P_{\beta}(\gamma_{s})= \widehat{W}(\gamma_{s}^{-1})
=
\frac{\widehat{t_{u_s }}}{\gamma_{s}^{u_s }}
+\frac{\widehat{t_{u_s +1}}}{\gamma_{s}^{u_s +1}}
+\frac{\widehat{t_{u_s +2}}}{\gamma_{s}^{u_s +2}}
+\ldots,\qquad {\rm with}~
|\widehat{t_{j}}| \leq m, j=u_s, u_s +1, \ldots$$
with $\,\widehat{t_{u_s }} \neq 0$.
The exponent $u_s$ appearing
in the first term  
is 
defined in \cite{frougnypelantovasvobodova};
by
Theorem 4, Remarks 5 to 7,
in \cite{frougnypelantovasvobodova},
there exists a positive real
number $\kappa_{\gamma_s , \mathcal{B}} > 0$
such that $u_s$ is the minimal integer 
such that
$$\gamma_{s}^{u_s -1} \geq 
\frac{\kappa_{\gamma,_s  \mathcal{B}}}
{|P(\gamma_s )|} .$$
Since $\lim_{s \to \infty} \gamma_s
= \beta > 1$ and that the alphabet
$\mathcal{B}$ does not depend upon
$s$, from Theorem 4, Remarks 5 to 7,
in \cite{frougnypelantovasvobodova},
we can replace 
$\kappa_{\gamma,_s  \mathcal{B}}$
by a constant $\kappa > 0$,
independent of $s$
(cf also \cite{dutykhvergergaugry2}). 
Thus
$$\lim_{s \to \infty} u_s = +\infty.$$

The 
sequence $(u_{s })$ is here defined by the bounds
\begin{equation}
\label{bound_usA}
\bigl|
\beta^{u_{s} - 1} (P_{\beta}(\gamma_{s}^{-1}))
\bigr|
\geq
\bigl|
\gamma_{s}^{u_s -1} (P_{\beta}(\gamma_{s}^{-1}))
\bigr| \geq \kappa
\end{equation}
and
$$\bigl|
\gamma_{s}^{u_s} (P_{\beta}(\gamma_{s}^{-1}))
\bigr| \leq 
\frac{m}{1-\gamma_{s}^{-1}}.$$
To the collection $(\gamma_s)_{s \geq W_{\nu}}$
is associated the collection
$(\sigma_{s}: \gamma_{s} \to r_s)_{s \geq W_{\nu}}$ 
of $\qb$-automorphisms of $\cb$.
Now, for any $s \geq W_{\nu}$, 
let us conjugate the eventually periodic
representation of
$-P_{\beta}(\gamma_{s}^{-1})$
by $\sigma_{s}$, 
term by term.
For the cases (i) and (ii-1)
we consider
\eqref{boutgammashift_uZERO},
and in the case (ii-2) we consider
\eqref{hgammasSeries_alphabetA}.

\begin{lemma}
\label{majoration}
Denote $c_{lent} = \frac{\pi |z_{1,n}|}{\, a_{\max}}$. 
Then
$$s \geq W_{\nu} \qquad 
\Longrightarrow
\qquad
|r_s| < 1- \frac{c_{lent}}{n}.$$
\end{lemma}
\begin{proof}
The inequality readily comes from the proof
of Theorem \ref{_cercleoptiSALEM} since the condition of 
Rouch\'e holds true.
\end{proof}

Using Lemma \ref{majoration}
we deduce:

\noindent
{\bf Case (i) and (ii-1):} with the minimal alphabet
$\{-1,0,+1\}$,
\begin{equation}
\label{zerolenticularMAJO}
\bigl|
P_{\beta}(r_{s})
\bigr|
\leq
| r_{s} |^{u_s} \frac{1}{1 - |r_s|}
\leq
\frac{n}{c_{lent}}
(1-\frac{c_{lent}}{n})^{u_s}
,
\end{equation}

\noindent
{\bf Case (ii-2):} with the alphabet
$\mathcal{B} =
\{-m, \ldots, +m\}$,
\begin{equation}
\label{zerolenticularMAJO_A}
\bigl|
P_{\beta}(r_{s})
\bigr|
\leq
| r_{s} |^{u_s} \frac{m}{1 - |r_s|}
\leq
\frac{n \, m}{c_{lent}}
(1-\frac{c_{lent}}{n})^{u_s } .
\end{equation}
In both cases, $\lim_{s \to \infty} u_s =+\infty$.
The rhs of \eqref{zerolenticularMAJO},
resp.
of \eqref{zerolenticularMAJO_A},
tends to 0 if $s$ tends to infinity.
We have: $\lim_{s \to \infty}
P_{\beta}(r_{s}) = 0$.
But $\omega_{1,n} = \lim_{s \to \infty} r_s$
and $z \to
P_{\beta}(z)$ is continuous.
Therefore there exists $s_0 \geq W_{\nu}$ 
such that $s \geq s_0 \Longrightarrow
\bigl|
P_{\beta}(r_{s})
\bigr| < \nu/2$. Contradiction.

\

The only limit possibility
is $P_{\beta}(\omega_{1,n})=0$.
To summarize,

$$ f_{\beta}(\omega_{1,n}) = 0 \qquad \Longrightarrow
\qquad
P_{\beta}(\omega_{1,n}) =0.$$

\end{document}